\documentclass{amsart}[11pt,fullpage]

\usepackage{amsmath,amssymb,amsthm}
\usepackage{mathrsfs}
\usepackage{pxfonts}

\usepackage{tikz}
\usepackage{multirow}

\usepackage{mycommands}
\newtheorem*{ack}{Acknowledgments}

\newcommand{\mr}{\text{mr}}

\newcommand{\ee}{\overset{\epsilon}{=}}
\newcommand{\aug}{\text{Aug}}
\newcommand{\augrep}{KCH representation }
\newcommand{\augreps}{KCH representations }                  

\begin{document}

\title[]{Knot contact homology and representations of knot groups}
\author[]{Christopher R. Cornwell}
\address[]{Mathematics Department, Duke University, Box 90320, Durham, NC 27708}
\email[]{cornwell@math.duke.edu}

\begin{abstract}
We study certain linear representations of the knot group that induce augmentations in knot contact homology. This perspective enhances our understanding of the relationship between the augmentation polynomial and the $A$-polynomial of a knot. For example, we show that for 2-bridge knots the polynomials agree and that this is never the case for (non-2-bridge) torus knots, nor for a family of 3-bridge pretzel knots. In addition, we obtain a lower bound on the meridional rank of the knot. As a consequence, our results give a new proof that torus knots and a family of pretzel knots have meridional rank equal to their bridge number.
\end{abstract}

\maketitle

\vspace*{-1cm}
\section{Introduction}

Let $K$ be a knot in $\rls^3$. The knot contact homology $HC_*(K)$ of $K$ is the Legendrian contact homology of a Legendrian torus over $K$ in the unit cotangent bundle of $\rls^3$. It appears that $HC_*(K)$ is both a very powerful invariant and also reasonably computable (see \cite{EENS12} and \cite{Ng11HT}). It was found in \cite{Ng08} that augmentations of the DGA underlying $HC_*(K)$ are related to the $A$-polynomial introduced in \cite{CCGLS}: the $A$-polynomial divides the augmentation polynomial $\aug_K(\lambda,\mu^2)$. 

This paper begins an effort to better understand this relationship through representations of the knot group, and specifically to understand factors of $\aug_K(\lambda,\mu)$ failing to appear in the $A$-polynomial. The particular representations we study, called \emph{KCH representations} (see Definition \ref{DefKCHrep}), have image in $GL_n\bb C$ and satisfy a certain condition on the peripheral subgroup.

We work with a specialization of $HC_*(K)$ obtained by setting an element in the ground ring to 1. The result is an algebra over the polynomial ring $R_0=\ints[\lambda^{\pm1},\mu^{\pm1}]$ (see Section \ref{KCHreview} for a definition, and \cite{NgSurv12} for a discussion of the more general invariant). An \emph{augmentation} of $HC_*(K)$ is equivalent to an algebra map $\epsilon:HC_0(K)\to\bb C$. The augmentation polynomial $\aug_K(\lambda,\mu)$ is defined so that its zero locus is the complex curve that is the closure of points $\setn{(\epsilon(\lambda),\epsilon(\mu))}{\epsilon\text{ an augmentation of }HC_*(K)}$. 

Throughout we use $\pi_K$ to denote the fundamental group of $\overline{\rls^3\setminus n(K)}$ and write $\bb C^*$ for $\bb C\setminus\{0\}$.

\subsection{Results} We first address the dimension of irreducible \augreps of $\pi_K$. If $\set{g_1,g_2,\ldots,g_r}$ is a set of generators for $\pi_K$ with the property that $g_i$ is a meridian of $K$ for each $1\le i\le r$, we call the generating set \emph{meridional}. 

\begin{thm} Let $\set{g_1,\ldots,g_r}$ be a meridional generating set for $\pi_K$. If $\rho:\pi_K\to\text{GL}_n\bb C$ is an irreducible \augrep of $\pi_K$ then $n\le r$.
\label{ThmDimBound}
\end{thm}

The \emph{meridional rank} of $K$, denoted $\mr(K)$, is the minimal size of a meridional generating set. Theorem \ref{ThmDimBound} produces a lower bound on $\mr(K)$. We will see that this bound is sharp on torus knots and on the $(-2,3,2k+1)$ pretzel knots (see Section \ref{SecTorusknots}). 

The meridional rank is bounded above by the bridge number of a knot. Problem 1.11 of \cite{Kir95}, which is attributed to Cappell and Shaneson, asks whether every knot with meridional rank $n$ is an $n$-bridge knot. The lower bound from Theorem \ref{ThmDimBound} suggests that \augreps might be used to study this problem.

A KCH representation $\rho:\pi_K\to\text{GL}_n\bb C$ induces an augmentation $\epsilon$ of $HC_*(K)$ (see Section \ref{SubsecAugKCHReps}). Moreover, \augreps can be made to have image in SL$_n\bb C$, and from this one sees that the $A$-polynomial divides $\aug_K(\lambda,\mu^2)$ (see section \ref{SecAugReps}; cf.\cite{Ng08}). For $m\in\pi_K$ a meridian of $K$, and $\ell\in\pi_K$ a 0-framed longitude, the values $\epsilon(\mu)$ and $\epsilon(\lambda)$ of the induced augmentation are particular eigenvalues of $\rho(m)$ and $\rho(\ell)$, respectively. With this in mind, we use $n$-dimensional \augreps to define for each $n\ge2$ the \emph{$n$-dimensional $A$-polynomial} $A_K^n(\lambda,\mu)$ of $K$, with zero locus being the closure of points in $(\bb C^*)^2$ given by these eigenvalues. 

If $\rho:\pi_K\to\text{GL}_n\bb C$ is a reducible KCH representation, then there is a KCH representation of lower degree that corresponds to the same point $(\epsilon(\mu),\epsilon(\lambda))$ in the curve of eigenvalues (see Lemma \ref{KCHIrreps}). Hence Theorem \ref{ThmDimBound} implies that the polynomials $A_K^n(\lambda,\mu)$ stabilize for sufficiently large $n$. Define the result to be the \emph{stable $A$-polynomial} $\wt{A}_K(\lambda,\mu)$.\footnote{The term \emph{stable $A$-polynomial} was coined by Stavros Garoufalidis in private communications.} We have that $\wt{A}_K(\lambda,\mu)$ divides $\aug_K(\lambda,\mu)$.

Remarkably, for any knot $K$, all augmentations (with $\epsilon(\mu)\ne1$) arise from a KCH representation. The proof of this fact will appear in a paper that is in preparation \cite{C13b}. In Section \ref{Sec2bridge} we provide a proof in the case that $K$ is a 2-bridge knot, using results developed by Riley on representations of 2-bridge knots.

\begin{thm} Let $K$ be a 2-bridge knot and $\epsilon:HC_0(K)\to\bb C$ an augmentation with $\epsilon(\mu)\ne1$. Then $\epsilon$ is induced from a \augrep $\rho_{\epsilon}:\pi_K\to\text{GL}_2\bb C$.
\label{Thm2bridge}
\end{thm}

Theorem \ref{Thm2bridge} confirms a conjecture of Ng \cite{Ng08}, that $\aug_K(\lambda,\mu^2)$ is equal to $(1-\mu^2)A_K(\lambda,\mu)$ when $K$ is a 2-bridge knot, where $A_K(\lambda,\mu)$ is the $A$-polynomial of \cite{CCGLS}. On the other hand, torus knots exhibit that there can be irreducible \augreps of high degree.

\begin{thm}Given $1\le p<q$, with $p,q$ relatively prime, let $T(p,q)$ denote the $(p,q)${--}torus knot and $\pi$ its knot group. For every $1\le n\le p$ and each $\mu_0\in\bb C^*$ there is a degree $n$ irreducible \augrep of $\pi$ with $\mu_0$ as an eigenvalue of the meridian $m$. In fact, 
	\[\wt{A}_{T(p,q)}=(\lambda\mu^{pq-q}+(-1)^p)\prod_{n=1}^{p-1}(\lambda^n\mu^{(n-1)pq}-1).\]
\label{Thmtorusknot}
\end{thm}

We remark that the torus knot $T(p,q)$ may be put into bridge position with $\min(p,q)$ bridges. Theorem \ref{Thmtorusknot} produces irreducible \augreps in dimension $\min(p,q)$ and, by Theorem \ref{ThmDimBound}, this implies the meridional rank of $T(p,q)$ is the bridge number (cf. \cite{RZ}). 

Finally, we consider the family of $(-2,3,2k+1)$ pretzel knots, which have a projection as shown in Figure \ref{FigPretzKnot}. The bridge number of these pretzel knots is equal to 3 if $k\ne-1,0$.

\begin{thm} Let $K$ be the $(-2,3,2k+1)$ pretzel knot, where $k\ne -1,0$. Then there is a 3-dimensional, irreducible \augrep of $\pi_K$. Furthermore, $(1-\lambda\mu^{2k+6})$ divides $A_K^3(\lambda,\mu)$.
\label{ThmPretzels}
\end{thm}

As a consequence every knot $K$ in this family has meridional rank at least 3, and so the meridional rank and bridge number of the $(-2,3,2k+1)$ pretzel knot agree (cf. \cite{BZ}).

The paper is organized as follows. In Section \ref{KCHreview} we provide a quick review of the background on knot contact homology, recalling an interpretation of the degree zero part as the cord algebra. In Section \ref{SecAugReps} we define the augmentation polynomial and KCH representations, use the cord algebra to see how \augreps induce an augmentation of $HC_*(K)$, and indicate how this produces a factor of the augmentation polynomial equal to the $A$-polynomial. We also prove the dimension bound in Theorem \ref{ThmDimBound}. Section \ref{Sec2bridge} is devoted to the proof of Theorem \ref{Thm2bridge}. In Section \ref{SecTorusknots} we prove Theorem \ref{Thmtorusknot} and Theorem \ref{ThmPretzels}.

\begin{ack} The author would like to thank Lenny Ng for his guidance and for many very helpful conversations.
\end{ack}

\section{Knot contact homology and the cord algebra}
\label{KCHreview}

\subsection{Knot contact homology}
\label{SubsecKCH}

The focus of this paper is on \augreps and the augmentations they induce on a specialization of knot contact homology. Using results of \cite{EENS12}, \cite{Ng08}, and \cite{Ng11HT} augmentations on this specialization may be understood purely through the cord algebra (using the formulation in Theorem \ref{ThmPi1Cords}) which has a topological definition. The reader who so chooses may, in fact, skip the review of the definition of $HC_*(K)$ via Legendrian contact homology and begin with the statement of Theorem \ref{ThmPi1Cords}, taking it as a definition of $HC_0(K)$. 

We define knot contact homology through the tool of Legendrian contact homology (LCH), introduced by Eliashberg and Hofer \cite{Eli98}. Many details on LCH are omitted. 

Being pertinent to the setting of knot contact homology of knots in $\rls^3$, we only discuss contact manifolds $P\times\rls$, where $z$ is the $\rls$ coordinate, $P$ has an exact symplectic form $d\lambda$, and the contact 1-form is $dz-\lambda$. In fact, we are interested in the case of the 1-jet space $J^1(M)$, topologically equal to $T^*M\times\rls$, where $\lambda$ is the canonical 1-form. The proof of the invariance of LCH in $(P\times\rls,\ker(dz-\lambda))$ was carried out in \cite{EES07}.

Let $\Lambda$ be a Legendrian in $(V=T^*M\times\rls,\ker(dz-\lambda))$. A \emph{Reeb chord} is an integral curve to the Reeb field $\bd_z$ that starts and ends on $\Lambda$. We assume that $\Lambda$ has trivial Maslov class and finitely many Reeb chords $a_1,\ldots, a_n$. Let $R$ denote the group ring of $H_2(V,\Lambda)$. 

Define $\cl A$ to be the noncommutative unital algebra over $R$ freely generated by $a_1,\ldots,a_n$. The grading on generators is by the Conley-Zehnder index (minus 1), and the base ring has grading 0. Extend the grading to all of $\cl A$ in the usual way.

The differential $\bd$ on a generator $a_i$ is defined by counting holomorphic disks in a moduli space $\scr M(a_i;a_{j_1},\ldots,a_{j_k})$ modulo an $\rls$-action. After declaring $\bd(r)=0$ for all $r\in R$ and defining $\bd(a_i)$ for each generator, extend $\bd$ to $\cl A$ via the signed Leibniz rule. The definition of $\bd$ on a generator $a_i$ is
\[\bd(a_i)=\sum_{\dim\scr M(a_i;a_{j_1},\ldots,a_{j_k})/\rls=0}\sum_{\D\in\scr M/\rls}\text{sgn}(\D)e^{[\D]}a_{j_1}\cdots a_{j_k}.\]
In the above definition, $\scr M(a_i;a_{j_1},\ldots,a_{j_k})$ is the moduli space of holomorphic disks $\D$ in the symplectization of $V$ (the almost complex structure being compatible and suitably generic) with the following boundary conditions: 
	\en{
	\item[-] The boundary of $\D$ is in the Lagrangian $\rls\times\Lambda$; 
	\item[-] $\D$ has $k+1$ ends, one being asymptotic to $a_i$ at $+\infty$, the other $k$ asymptotic to $a_{j_1},\ldots,a_{j_k}$ at $-\infty$ and in that order as determined by the oriented boundary of the disk; 
	\item[-] sgn$(\D)$ is an orientation sign on $\D$, and $[\D]$ is the homology class of $\D$ in $H_2(V,\Lambda)$ (using a chosen disk for each Reeb chord in $V$, with part boundary on the chord, to cap off $\D$).
	}

By work in \cite{EES07}, $\bd$ makes $\cl A$ into a differential graded algebra and $H_*(\cl A,\bd)$ is a Legendrian invariant of $\Lambda$. 

Given a knot $K$ in $\rls^3$, the unit cotangent bundle $ST^*\rls^3$ with its canonical contact structure is contactomorphic to $J^1(S^2)$. Define $\Lambda_K$ to be the unit conormal bundle over $K$, that is
	\[\Lambda_K=\setn{(q,p)\in ST^*\rls^3}{q\in K\ \text{and}\ \gen{p,v}=0\text{ for }v\in T_qK}.\]
We obtain a Legendrian torus $\Lambda_K$ in $V=ST^*\rls^3$. Define the knot DGA $(\cl A_K,\bd_K)$ to be the differential graded algebra of $\Lambda_K$ described above. Knot contact homology $HC_*(K)$ is defined as the Legendrian contact homology $H_*(\cl A_K,\bd_K)$. Isotopy on $K$ lifts to Legendrian isotopy on $\Lambda_K$, and so $HC_*(K)$ is an invariant of the knot type $K$. 

There is an isomorphsism $H_2(ST^*\rls^3,\Lambda_K)\cong H_2(S^2)\oplus H_1(T^2)$, through which we identify $R$ with the ring $\ints[U^{\pm1},\lambda^{\pm1},\mu^{\pm1}]$ where $U$ is a generator of $H_2(S^2)$ and $\lambda$ (resp. $\mu$) corresponds to a longitude (resp. meridian) of $K$.

This paper focuses on a specialization of $(\cl A_K,\bd_K)$, where multiplication by $U$ is trivial. In this specialization, the degree zero homology $H_0(\cl A_K,\bd_K)$ admits a useful topological description called the cord algebra \cite{Ng08}. Were this story to be paralleled with the full coefficient ring, it would be a hopeful step toward understanding recent computations that show coincidences between the three-variable augmentation polynomial, super-$A$ polynomials in string theory, and colored HOMFLY polynomials. 

Specialize the knot DGA so the differential $\bd_K$ only accounts for the homology of $\bd\D\in H_1(\Lambda_K)$. In other words, we take $U=1$ in the base ring. Unless otherwise specified the specialization to $U=1$ is assumed for the remainder of the paper, yet we leave our notation unaltered.

The (framed) knot DGA in \cite{Ng08} was shown to be an invariant of $K$ independently of contact homology, and by \cite{EENS12} is a calculation of $HC_*(K)$ (with $U=1$) that uses a presentation of $K$ as the closure of a braid. We will define the degree zero part of this DGA in order to arrive at the cord algebra, which in turn will lead us to KCH representations. See \cite{NgSurv12} and \cite{Ng11HT} for a definition of the full DGA from this perspective.

Given $n\ge1$, let $A_n$ be the noncommutative unital algebra over $\ints$ freely generated by $n(n-1)$ elements $a_{ij}$, $1\le i\ne j\le n$. Let $B_n$ denote the braid group on $n$ strands. If $\sg_k$ is one of the standard generators, then define $\phi:B_n\to\text{Aut }A_n$ by defining it on each generator as
	\[\phi_{\sg_k}:\begin{cases}
				a_{ij}\mapsto a_{ij},					& i,j\ne k,k+1\\
				a_{k+1,i}\mapsto a_{ki},				&  i\ne k,k+1\\
				a_{i,k+1}\mapsto a_{ik},				& i\ne k,k+1\\
				a_{k,k+1}\mapsto -a_{k+1,k},			&\ \\
				a_{k+1,k}\mapsto -a_{k,k+1},			&\ \\
				a_{ki}\mapsto a_{k+1,i}-a_{k+1,k}a_{ki}	&i\ne k,k+1\\
				a_{ik}\mapsto a_{i,k+1}-a_{ik}a_{k,k+1}	&i\ne k,k+1
				\end{cases}\]

Under an identification of $A_n$ with a certain algebra of homotopy classes of arcs in a punctured disk, the representation $\phi$ corresponds to the mapping class group action (see \cite{NgSurv12}). Inject $\iota:B_n\hookrightarrow B_{n+1}$ by letting the last strand not interact and for $B\in B_n$ let $\phi_B^*=\phi_B\circ\iota$. We may define two matrices $\Phi_B^L, \Phi_B^R\in\text{Mat}_{n\times n}(A_n)$ by
\[\phi_B^*(a_{i,n+1})=\sum_{j=1}^n(\Phi_B^L)_{ij}a_{j,n+1},\]
\[\phi_B^*(a_{n+1,i})=\sum_{j=1}^na_{n+1,j}(\Phi_B^R)_{ji}.\]

Finally, define the matrices ${\bf A}$ by 
\[{\bf A_{ij}}=\begin{cases}a_{ij}	& i<j\\
						-\mu a_{ij}	& i>j\\
						1-\mu		& i=j\end{cases},
\comment{
{\bf\hat A_{ij}}=\begin{cases}Ua_{ij}	& i<j\\
						-\mu a_{ij}	& i>j\\
						U-\mu		& i=j\end{cases},
						}\]
and the diagonal matrix ${\bf\Lambda}=\text{diag}[\lambda\mu^w,1,\ldots,1]$, where $w$ is the writhe of the braid $B\in B_n$. The following gives a concrete algebraic interpretation of $HC_0(K)$.

\begin{thm}[see \cite{NgSurv12}, \S 4] Let $R_0=\ints[\lambda^{\pm1},\mu^{\pm1}]$ and suppose $K$ is the closure of a braid $B\in B_n$. Then
	\[HC_0(K)\cong (A_n\otimes R_0)\big/\cl I,\]
	where $\cl I$ is the ideal generated by entries in the matrices ${\bf A}-{\bf\Lambda}\cdot\phi_B({\bf A})\cdot{\bf\Lambda}^{-1}, {\bf A}-{\bf\Lambda}\cdot\Phi^L_B\cdot{\bf A}$, and ${\bf A}-{\bf A}\cdot\Phi^R_B\cdot{\bf\Lambda}^{-1}$.
\label{ThmHC0}
\end{thm}

\subsection{The cord algebra}
\label{SubsecCordAlg}

Let $K\subset S^3$ be an oriented knot with a basepoint $\ast\in K$. A \emph{cord} of $(K,\ast)$ is a path $\gamma:[0,1]\to S^3$ with $\gamma^{-1}(\ast)=\emptyset$ and $\gamma^{-1}(K)=\set{0,1}$. Define the \emph{cord algebra} $\cl C_K$ to be the unital tensor algebra over $R_0$ freely generated by homotopy classes of cords (note that endpoints of cords may move in the homotopy, if they avoid $\ast$), modulo the ideal generated by the relations:

	\begin{tikzpicture}[scale=0.5,>=stealth]
	\draw(-3,0) node {(1)}; 
	\draw[very thick,->](1,1)--(-1,-1);
	\draw[gray,->](0,0) ..controls (-1.5,0)and(0,1.5).. (0,0);
	\draw(2,-.25) node {$=$};
	\draw(3.5,-.25) node {$1-\mu$;};

	\draw(-3,-3) node {(2)}; 
	\draw[very thick,->](1,-2)--(-1,-4);
	\draw[gray,->](0.35,0.35-3) -- (-1,-2);
	\filldraw(0,-3) circle (3pt);
	\draw[very thick](.4,-3.4) node {{\huge$\ast$}};
	
	\draw(2,-.25-3) node {$=$};

	\draw(3,-.25-3) node {$\lambda$};
	\draw[very thick,->](5.5,-2)--(3.5,-4);
	\draw[gray,->](4.15,-0.35-3) -- (3.5,-2);
	\filldraw(4.5,-3) circle (3pt);
	\draw[very thick](4.9,-3.4) node {{\huge$\ast$}};

	\draw (6.5,-.25-3) node {and};
	\draw[very thick,->](9+1,-2)--(9-1,-4);
	\draw[gray,<-](9-0.35,-0.35-3) -- (9-1,-2);
	\filldraw(9,-3) circle (3pt);
	\draw[very thick](9.4,-3.4) node {{\huge$\ast$}};
	
	\draw(9+2,-.25-3) node {$=$};

	\draw(9+3,-.25-3) node {$\lambda$};
	\draw[very thick,->](9+5.5,-2)--(9+3.5,-4);
	\draw[gray,<-](9+4.85,0.35-3) -- (9+3.5,-2);
	\filldraw(9+4.5,-3) circle (3pt);
	\draw[very thick](9+4.9,-3.4) node {{\huge$\ast$}};
	
	\draw(-3,-6) node {(3)}; 
	\draw[very thick,->](1,-5)--(-1,-7);
	\draw[->, white,line width=3pt](-1,-5)--(1,-7);
	\draw[->, gray](-1,-5)--(1,-7);	
	\draw(2,-.25-6) node {$-$};
	\draw(3,-.25-6) node {$\mu$};
	\draw[->,gray](3.5,-5)--(5.5,-7);
	\draw[->,very thick,draw=white,double=black,double distance=1pt](5.5,-5)--(3.5,-7);
	\draw(6.5,-.25-6)node{$=$};
	\draw[very thick,->](9,-5)--(7,-7);
	\draw[gray,->](7,-5)--(8,-6);
	\filldraw(9.5,-6) circle (1pt);
	\draw[very thick,->](12,-5)--(10,-7);
	\draw[gray,->](11,-6)--(12,-7);
	
	\end{tikzpicture}

\begin{rem} In the relations on $\cl C_K$ depicted above the thicker, dark curves are part of the knot $K$ and the thinner curves are part of a cord. Moreover, the relations are understood to be in $\rls^3$ and not just as planar diagrams.
\end{rem}

The cord algebra $\cl C_K$ is manifestly a knot invariant. Moreover, the isomorphism in Theorem \ref{ThmHC0} was used in \cite{Ng08} to prove that the cord algebra is isomorphic to $HC_0(K)$.
	
\begin{thm}[\cite{Ng08}] $\cl C_K$ is isomorphic as an $R_0$ algebra to $(A_n\otimes R_0)\big/\cl I$, and thus to $HC_0(K)$.
\label{ThmCords}
\end{thm}

Finally, it is useful to reformulate the cord algebra $\cl C_K$ in terms of the knot group. Given a cord in $\cl C_K$, and $X_K=\overline{\rls^3\setminus n(K)}$ the complement of $K$, push the endpoints slightly off of $K$ so the cord is in $X_K$ and connect the endpoints via a longitudinal curve on $\bd n(K)$ to get an element of $\pi_K=\pi_1(X_K)$.

\begin{thm}[\cite{Ng08}] Let $P_K$ denote the underlying set of the knot group $\pi_K$, where we write $[\gamma]\in P_K$ for $\gamma\in\pi_K$. Let $e$ denote the identity, $m$ a choice of meridian, and $\ell$ the 0-framed longitude of $K$. The cord algebra $HC_0(K)$ is isomorphic to the noncommutative unital algebra freely generated by elements of $P_K$ over $R_0=\ints[\lambda^{\pm1},\mu^{\pm1}]$ modulo the relations:
	\en{
	\item $[e]=1-\mu$;
	\item $[m\gamma]=[\gamma m]=\mu[\gamma]$ and $[\ell\gamma]=[\gamma\ell]=\lambda[\gamma]$, for any $\gamma\in\pi_K$;
	\item $[\gamma_1\gamma_2]-[\gamma_1m\gamma_2]=[\gamma_1][\gamma_2]$ for any $\gamma_1,\gamma_2\in\pi_K$.
	}
\label{ThmPi1Cords}
\end{thm}

It is this last interpretation of $HC_0(K)$ that will indicate how \augreps induce an augmentation.

\section{\augreps and dimension bounds}
\label{SecAugReps}

\subsection{Augmentations of $HC_0(K)$ and \augreps}
\label{SubsecAugKCHReps}

Given a unital ring $S$, an \emph{$S$-augmentation} of $(\cl A_K,\bd_K)$ is a graded algebra map $\epsilon:\cl A_K\to S$ such that $\epsilon\circ\bd=0$. Since $\cl A_K$ is supported in non-negative grading, such maps are in bijective correspondence with algebra maps $\epsilon:HC_0(K)\to S$. We fix $S=\bb C$ and refer to a $\bb C$-augmentation of $(\cl A_K,\bd_K)$ simply as an augmentation of $HC_*(K)$.

Consider the set in $(\bb C^*)^2$ given by
\[V_K=\setn{(\epsilon(\lambda),\epsilon(\mu))}{\epsilon\text{ is an augmentation of }HC_*(K)}.\]

\begin{conj} The maximum dimensional component of the closure of $V_K$ is 1-dimensional for all knots $K\subset\rls^3$.
\label{VK}
\end{conj}

The closure of $V_K$ contains a 1-dimensional component (the lines $\mu=1$ and $\lambda=1$; see below). If Conjecture \ref{VK} holds and there is no 2-dimensional component, let the \emph{augmentation polynomial} $\aug_K(\lambda,\mu)$ be the corresponding reduced polynomial, where \emph{reduced} means that $\aug_K(\lambda,\mu)$ has no repeated factors. This only defines $\aug_K(\lambda,\mu)$ up to a unit in $\bb C[\lambda^{\pm1},\mu^{\pm1}]$; however, we may take $\aug_K(\lambda,\mu)\in\ints[\lambda^{\pm1},\mu^{\pm1}]$ (and with coefficients that are coprime), as $\bd_K$ involves only $\ints$ coefficients. Further, we can make the polynomial not divisible by $\lambda,\mu$ (and contain no negative exponents of $\lambda,\mu$) by multiplying by an appropriate power of $\lambda$ and $\mu$. The result is well-defined up to overall sign. 

We are now prepared to discuss \augreps and their relation to augmentations. We use notation for the knot group and its elements that agrees with the statement of Theorem \ref{ThmPi1Cords}. 

\begin{defn} Let $V$ be a complex vector space with dimension $n$. We say that a homomorphism $\rho:\pi_K\to\text{GL}(V)$ is a \emph{\augrep of $\pi_K$ of degree $n$} if for a meridian $m$ of $K$, the map $\rho(m)$ is diagonalizable and has 1 as an eigenvalue with multiplicity $n-1$. We call $\rho$ a \emph{KCH irrep} if it is irreducible as a representation.
\label{DefKCHrep}
\end{defn}

We often identify $V$ with $\bb C^n$ via a basis $\{e_1,\ldots,e_n\}$ of eigenvectors of $\rho(m)$. In this notation $e_1$ denotes the eigenvector with eigenvalue $\mu_0\ne1$. Note that the diagonal matrix $\text{diag}[\mu_0,1,\ldots,1]$ represents $\rho(m)$ in this basis. 

\begin{rem}When $\pi_K$ is generated by two elements $\{m,g\}$, and $m$ is a meridian, then for KCH irreps a judicious choice of basis makes the matrix for $\rho(g)$ simple. Starting with $\{e_1,\ldots,e_n\}$, note that $\rho(g)e_1$ is not in $\bb C\langle e_1\rangle$. Subtract the $e_1$ part to get $e_2'=\rho(g)e_1-x_1e_1$, an eigenvector of $\rho(m)$ with eigenvalue 1. Continue, defining $e_i'$ by subtracting the $e_1$ part from $\rho(g)e_{i-1}'$. The irreducibility of $\rho$ guarantees $\{e_1,e_2',\ldots,e_n'\}$ is a basis. Furthermore $\rho(m)=\text{diag}[\mu_0,1,\ldots,1]$ and $$\rho(g)=\begin{pmatrix}x_1&\ldots&x_n \\ & \text{Id}_{n-1} & \vdots \\ & & x_{2n+1}\end{pmatrix}$$ in this basis.
\label{Rem2genForm}
\end{rem}

As $\ell$ and $m$ commute, $e_1$ is an eigenvector for $\rho(\ell)$. Write $\lambda_0$ for the corresponding eigenvalue. Identify $HC_0(K)$ with the algebra described in Theorem \ref{ThmPi1Cords}.

\begin{defn} With notation as above, let $\rho:\pi_K\to\text{GL}(V)$ be a \augrep of degree $n$ with an inner product $\langle\ ,\ \rangle$ on $V$ so that $\set{e_1,\ldots,e_n}$ is orthonormal. Define $\epsilon_\rho(\mu)=\mu_0$, $\epsilon_\rho(\lambda)=\lambda_0$, and $\epsilon_\rho([\gamma]) = (1-\mu_0)\langle\rho(\gamma)e_1,e_1\rangle$ for any $\gamma\in\pi_K$. This determines a map $\epsilon_\rho:HC_0(K)\to\bb C$ called the \emph{induced augmentation of $\rho$}.
\label{DefnInducedAug}
\end{defn}

\begin{thm}[\cite{NgSurv12}] For any \augrep $\rho:\pi_K\to\text{GL}(V)$ the induced augmentation is well-defined.
\label{ThmKCHReps}
\end{thm}

Theorem \ref{ThmKCHReps} appears as an exercise in \cite{NgSurv12}. The proof is as follows.

\begin{proof} 
We check that $\epsilon_\rho$ is well-defined by checking the relations of Theorem \ref{ThmPi1Cords}. Relation (1) is trivial since $\rho(e)$ is the identity. To check (2), note that, in the basis $\set{e_1,\ldots,e_n}$, the first row of the matrix for $\rho(m)\rho(\gamma)$ is $\mu_0$ times the first row of $\rho(\gamma)$. Thus for any $\gamma\in\pi_K$,
	\al{
	\epsilon_\rho([m\gamma])
	&=(1-\mu_0)\langle\rho(m)\rho(\gamma)e_1,e_1\rangle\\
	&=(1-\mu_0)\mu_0\langle\rho(\gamma)e_1,e_1\rangle\\
&=\epsilon_{\rho}(\mu[\gamma]),\text{ by definition.}
	}
Moreover, $\epsilon_{\rho}([\gamma m])=(1-\mu_0)\mu_0\langle\rho(\gamma)e_1,e_1\rangle$ since $e_1$ is a $\mu_0$-eigenvector. A similar argument holds for $[\ell\gamma],[\gamma\ell],$ and $\lambda[\gamma]$ since $\rho(\ell)$ commutes with $\rho(m)$. 

Finally, to see that $\epsilon_\rho$ respects (3), let $\gamma_1,\gamma_2\in\pi_K$. Note that $\rho(m)=I+(\mu_0-1)E_{11}$ where $I$ is the identity and $E_{11}$ the matrix with 1 in the top-left entry and 0 elsewhere. As a result $\rho(\gamma_1)\rho(m)\rho(\gamma_2)=\rho(\gamma_1)\rho(\gamma_2)+(\mu_0-1)\rho(\gamma_1)E_{11}\rho(\gamma_2)$, and hence
	\al{
	\epsilon_\rho([\gamma_1\gamma_2]-[\gamma_1 m\gamma_2])
&=(1-\mu_0)\langle\left(\rho(\gamma_1\gamma_2)-\rho(\gamma_1 m\gamma_2)\right)e_1,e_1\rangle\\	&=(1-\mu_0)^2\langle\rho(\gamma_1)E_{11}\rho(\gamma_2)e_1,e_1\rangle\\
&=(1-\mu_0)\langle\rho(\gamma_1)e_1,e_1\rangle(1-\mu_0)\langle\rho(\gamma_2)e_1,e_1\rangle\\
&=\epsilon_\rho([\gamma_1])\epsilon_\rho([\gamma_2]),
	}
the penultimate equality resulting from the image of $E_{11}\rho(\gamma_2)$ being in $\bb C\langle e_1\rangle$.
\end{proof}

\begin{rem} An augmentation $\epsilon:HC_0(K)\to\bb C$ satisfies $\epsilon([\gamma])=\mu_0^{\text{lk}(\gamma,K)}(1-\mu_0)$, where lk$(\gamma,K)$ is the linking number of $\gamma\in\pi_K$ with $K$, if and only if $\epsilon$ is induced from a 1-dimensional KCH representation $\rho:\pi_K\to\bb C$. This is checked by noting that $\rho(\gamma)=\mu_0^{\text{lk}(\gamma,K)}$ induces this augmentation and that this is the only abelian representation with $\rho(m)=\mu_0$.\label{RemKCH1D}
\end{rem}

Previous calculations of $\aug_K(\lambda,\mu)$ have been markedly difficult. The situation is much improved upon considering \augreps of $\pi_K$ as we will see in the succeeding section. There is a \emph{canonical augmentation} with $\epsilon(\mu)=1$ where $\epsilon(\lambda)$ may be any nonzero number \cite[Prop 5.6]{Ng08}. \augreps cannot account for this, since if $\mu_0=1$ then the representation is trivial as $\pi_K$ is normally generated by $m$. 

We recall the relationship found in \cite{Ng08} of $\aug_K(\lambda,\mu)$ to the $A$-polynomial $A_K(\lambda,\mu)$, namely
	
	\begin{equation}
	(\mu^2-1)A_K(\lambda,\mu)\Big\vert\aug_K(\lambda,\mu^2).
	\label{eqnAdividesAug}
	\end{equation}
	
The $A$-polynomial $A_K(\lambda,\mu)$ was introduced in \cite{CCGLS}; it was shown there that $\mu\pm1$ does not divide $A_K(\lambda,\mu)$. However, $\mu-1$ does always divide $\aug_K(\lambda,\mu)$ as seen by the canonical augmentation. 

Let $\rho:\pi_K\to\text{SL}_2\bb C$ denote a representation of $\pi_K$ into SL$_2\bb C$. If $m$ and $\ell$ denote a meridian and longitude of $K$ then $\rho(m)$ and $\rho(\ell)$ can be made simultaneously upper triangular:
	\[\rho(m)=\begin{pmatrix}\mu_1 &\ast\\0&\mu_1^{-1}\end{pmatrix}\qquad\rho(\ell)=\begin{pmatrix}\lambda_1 &\ast\\0&\lambda_1^{-1}\end{pmatrix}\]

The $A$-polynomial has zero locus equal to the top-dimensional component of the closure of points $(\lambda_1,\mu_1)\in(\bb C^*)^2$ such that there is a $\rho$ as above with $\lambda_1, \mu_1$ the upper-left entry of $\rho(\ell),\rho(m)$ respectively.

Given such a representation, let $\rho'(\gamma):=\mu_1^{\text{lk}(\gamma,K)}\rho(\gamma)$, the product of $\rho$ with the abelian representation. If $\mu_1\ne \pm1$ then $\rho(m)$ is diagonalizable and $\rho'$ defines a \augrep with $\mu_0=\mu_1^2$. Thus $A_K(\lambda,\mu)$ divides $\aug_K(\lambda,\mu^2)$.

\subsection{Irreducibility and dimension bounds}
\label{SubsecDimBounds}

It can occur that $\aug_K(\lambda,\mu)$ has more factors than those given in (\ref{eqnAdividesAug}). For example, Ng computes extra factors in $\aug_K(\lambda,\mu)$ for the $(3,4)$-torus knot and three pretzel knots \cite[end of \S 5]{Ng08}.

Similar to the definition of the variety $V_K$, consider the set
\[
U_K^n=\setn{(\epsilon(\lambda),\epsilon(\mu))}{\epsilon\text{ is induced from a degree $n$ KCH representation}}\subset(\bb C^*)^2.\]

From $U^n_K$ we define the $n$-dimensional $A$-polynomial $A_K^n(\lambda,\mu)$ in similar fashion to $\aug_K(\lambda,\mu)$. We would hope to understand $A_K^n(\lambda,\mu)$ by considering only KCH irreps with degree at most $n$, and may do so by the following result.

\begin{lem} Let $\rho:\pi_K\to\text{GL}(V)$ be a \augrep of degree $n$ that is reducible and corresponds to the point $(\lambda_0,\mu_0)$ in $U^n_K$. Then $(\lambda_0,\mu_0)\in U_K^{n'}$ for some $n'<n$. 
\label{KCHIrreps}
\end{lem}

\begin{proof}Let $V'\subset V$ be a proper invariant subspace. If $e_1$ is not contained in $V'$ then take the quotient representation $\overline\rho:\pi_K\to\text{GL}(V/V')$, which is well-defined as $V'$ is invariant. Consider a basis $\{e_1+V',f_2+V',\ldots,f_{n'}+V'\}$ of $V/V'$. Suppose that $f_i=\sum c^i_je_j$ as a vector in $V$. Then for $f_i'=f_i-c^i_1e_1$, the set $\{e_1+V',f_2'+V',\ldots,f_{n'}'+V'\}$ is also a basis of $V/V'$. As $\overline\rho(m)(f_i'+V')=f_i'+V'$ for $2\le i\le n'$, this basis shows $\overline\rho$ to be a KCH representation.

Since $e_1+V'$ is in the eigenvector basis for $\overline\rho(m)$ we get $\lambda_0,\mu_0$ in the induced augmentation.

If $e_1$ is contained in $V'$ define a new representation on $V'$ by restriction. Then $\rho(m)$ is diagonalizable in $V'$ so this is a KCH representation. Clearly the eigenvalues $\lambda_0,\mu_0$ are unchanged.
\end{proof}

We now bound the degree of a KCH irrep for a fixed knot $K$.

\vspace*{6pt}
\noindent{\bf Theorem \ref{ThmDimBound}.}\ \ {\it Let $\set{g_1,\ldots,g_r}$ be a meridional generating set for $\pi_K$. If $\rho:\pi_K\to\text{GL}_n\bb C$ is a KCH irrep of $\pi_K$ then $n\le r$.}
\vspace*{6pt}

\begin{proof}
Let $m$ be the meridian of $K$ sent by $\rho$ to $M:=\text{diag}[\mu_0,1,\ldots,1]$ (in the basis $\{e_1,e_2,\ldots,e_n\}$). For each $g_i$, let $w_i\in\pi_K$ be such that $g_i=w_imw_i^{-1}$. Write $W_{i}$ for the matrix representation of $\rho(w_{i})$. Now the specific form of $M$ guarantees that for each $1\le i\le r$, \[W_{i}MW_{i}^{-1}=I+(\mu_0-1)W_{i}E_{11}W_{i}^{-1}.\]
Now $(\mu_0-1)W_{i}E_{11}W_{i}^{-1}$ is a rank 1 matrix; let $v_{i}$ be a vector whose span equals the image. For each $1\le j\le r$ there is a scalar $\alpha_i^j\in\bb C$ for every $1\le i\le r$ such that	
\[\rho(g_i)v_j=v_j+\alpha_i^jv_i.\]
This shows that the span of $\set{v_1,v_2,\ldots,v_r}$ is an invariant subspace. We conclude $n\le r$ since $\rho$ is irreducible.
\end{proof}

\begin{cor} The sequence of polynomials $A_K^n(\lambda,\mu)$, $n\ge1$ stabilizes.
\label{CorStableA}
\end{cor}
\begin{proof} Combine the result of Lemma \ref{KCHIrreps} with Theorem \ref{ThmDimBound}.
\end{proof}

\section{The augmentation polynomial of 2-bridge knots}
\label{Sec2bridge}

Here we consider knots with a 2-bridge presentation, proving Theorem \ref{Thm2bridge}. Prior to this Lemma \ref{Lem2bridge} collects some identities on augmentations to be used in the proof. Studies of 2-bridge knots have been carried out in many settings. Of particular interest to our discussion are the papers \cite{Riley} and \cite{Sch}. 

Throughout Section \ref{Sec2bridge} and \ref{SecTorusknots} we derive and make use of a number of identities on the values of augmentations $\epsilon:HC_0(K)\to\bb C$. In an attempt to avoid a cluttering of $\epsilon$'s and demarcation symbols, we will adopt the notation
	\[X \ee Y \]
to denote that elements $X,Y$ of the cord algebra satisfy $\epsilon(X)=\epsilon(Y)$. At times $Y$ may instead be a complex number, in which case we mean that $\epsilon(X)=Y$. In addition, we drop the subscript of $\mu_0$ and simply write $\mu$ for both the element of $R_0$ and the eigenvalue. As in Section \ref{SecAugReps} the chosen meridian in the isomorphism of Theorem \ref{ThmPi1Cords} is denoted by $m$.

\begin{lem}
Given a knot $K$, consider any meridian $b$ and let $\epsilon:HC_0(K)\to \bb C$ be an augmentation of $K$ with $\epsilon(\mu)\ne1$. Then \begin{itemize}
			\item[(i)] $[b^{-1}] = \mu^{-1}([e]+[m]-[b])$ in $HC_0(K)$;
			\item[(ii)] for any $n\in\bb Z$, 
				\[\frac{[b^{n-1}][b]+([b]-[e])([m^{n-1}]-[b^{n-1}])}{1-\mu} \ee [b^n] \ee \frac{[b^{n+1}][b^{-1}]-\mu^{-1}([b]-[e])([m^{n+1}]-[b^{n+1}])}{1-\mu}.\]
	 \end{itemize}
\label{Lem2bridge}
\end{lem}

\begin{proof}
By Theorem \ref{ThmPi1Cords}, the identities $[e] = 1-\mu$ and also 
	\begin{equation}[g][h] = [gh]-[gmh]\qquad\text{and}\qquad	[m^{\pm1}g] = \mu^{\pm1}[g] = [gm^{\pm1}]
	\label{epsRel}
	\end{equation} 
hold in $HC_0(K)$, for any $g,h\in \pi_K$. We will often use the following reformulation of the equation on the left in (\ref{epsRel}): for $\delta=\pm1$ and any $g,h\in \pi_K$,
		\begin{equation} [gm^{\delta}h] = [gh] -\delta\mu^{\frac12(\delta-1)}[g][h].
		\label{SignIndptReln}
		\end{equation}
As $b$ is a meridian, choose $w\in\pi_K$ so that $b = wmw^{-1}$. Then (i) is derived by
			\al{
				[b]+\mu[b^{-1}]	&=[wmw^{-1}]+\mu[wm^{-1}w^{-1}]\\	&= [e]-[w][w^{-1}]+\mu([e]+\mu^{-1}[w][w^{-1}])\\
			&= [e]+[m].
			}
To check (ii), apply (\ref{SignIndptReln}) with $\delta=1$ and use that $\epsilon(1-\mu)\ne0$ to get
			\al{
				[b^n]	&= [wm^{n-1}w^{-1}]-[wm^{n-1}][w^{-1}]\\	
			&= [b^{n-1}]-\mu^{n-1}([e]-[b])\\	&\ee\frac1{1-\mu}\left([e][b^{n-1}]+[b][m^{n-1}]-[e][m^{n-1}]\right)\\
			&=\frac{[b^{n-1}][b]}{1-\mu}+\frac{([b]-[e])([m^{n-1}]-[b^{n-1}])}{1-\mu}.
			}
A similar calculation using (\ref{SignIndptReln}) with $\delta=-1$ gives
			\al{
				[b^n]	
						&\ee \frac{[b^{-1}][b^{n+1}]}{1-\mu}-\mu^{-1}\frac{([b]-[e])([m^{n+1}]-[b^{n+1}])}{1-\mu}.
			}
\end{proof}

We recall Schubert's normal form for a 2-bridge knot $K$ \cite{Sch} (see also \cite[Chp. 12]{BZKnots}). There are integers $p,q$ that determine $K$, where $-p<q<p$ and $p,q$ are coprime and odd when $K$ is a knot. Let $p-1=2k$. The group $\pi_K$ has presentation
		\[\pi_K \cong \gen{m,b\ \vert\ wm=bw},\]
where $w=m^{\ve_1}b^{\ve_2}\ldots m^{\ve_{2k-1}}b^{\ve_{2k}}$ and $\ve_i=(-1)^{\lfloor iq/p\rfloor}$.

\vspace*{6pt}
\noindent{\bf Theorem \ref{Thm2bridge}.}\ \ {\it Let $K$ be a 2-bridge knot and $\epsilon:HC_0(K)\to\bb C$ an augmentation with $\epsilon(\mu)\ne1$. Then $\epsilon$ is induced from a \augrep $\rho_{\epsilon}:\pi_K\to\text{GL}_2\bb C$. Thus, 
	\[\aug_K(\lambda,\mu^2)=(\mu^2-1)A_K(\lambda,\mu).\]
}
\vspace*{6pt}

\begin{proof} If $\epsilon([b])=\epsilon([m])$ then, as $b$ is conjugate to $m$, (\ref{epsRel}) implies $\epsilon([g])=\mu^{\text{lk}(g,K)}(1-\mu)$ for any $g\in\pi_K$. By Remark \ref{RemKCH1D}, $\epsilon$ must be induced from a degree 1 representation with $\rho(m)=\mu$. Summing with the trivial representation gives the result. We now assume $\epsilon([b])\ne\epsilon([m])$.

From the presentation of $\pi_K$ above and motivated by Remark \ref{Rem2genForm} and Lemma \ref{Lem2bridge} (i), define $\rho_{\epsilon}:\pi_K\to\text{GL}_2\bb C$ on the generators $m, b$ by
	\[\rho_{\epsilon}(m)=M=\begin{pmatrix}\mu &0\\ 0 & 1\end{pmatrix} \quad \text{and} \quad \rho_{\epsilon}(b)=B=\begin{pmatrix}\frac{\epsilon([b])}{1-\mu} & \frac{\epsilon([b]-[e])\epsilon([m]-[b])}{(1-\mu)^2}\\ 1 &\frac{\epsilon([e]+[m]-[b])}{1-\mu}\end{pmatrix}.\]

We proceed first to show that $\rho_{\epsilon}$ determines a well-defined representation, and second that it does induce the given augmentation $\epsilon$. The discussion in Section \ref{SubsecAugKCHReps} then implies that $\aug_K(\lambda,\mu^2)=(\mu^2-1)A_K(\lambda,\mu)$.

{\bf $\rho_\epsilon$ is well-defined:} Following Lemma 1 of \cite{Riley}, we find $U\in\text{SL}_2\bb C$ such that \begin{equation}UMU^{-1}=\begin{pmatrix}\mu&1\\0&1\end{pmatrix} \quad\text{and}\quad UBU^{-1}=\begin{pmatrix}\mu&0\\-\mu u&1\end{pmatrix}
	\label{newMB}
	\end{equation}
for some $u$ in the field $\bb Q(\mu,\epsilon([b]))$. Indeed, choose a square root $z$ of $\epsilon([m]-[b])\in\bb C$ and define 
	\[U=\begin{pmatrix}\frac{1}{z}&\frac{z}{1-\mu}\\ 0 & z\end{pmatrix}.\]
Then $UMU^{-1}$ and $UBU^{-1}$ have the desired form, where $-\mu u=\epsilon([m]-[b])$.

Recall the element $w=m^{\ve_1}b^{\ve_2}\ldots m^{\ve_{2k-1}}b^{\ve_{2k}}$ appearing in our presentation of $\pi_K$. We write $N=UMU^{-1}$ and $C=UBU^{-1}$ for the matrices in (\ref{newMB}) and write $W=N^{\ve_1}C^{\ve_2}\ldots N^{\ve_{2k-1}}C^{\ve_{2k}}$. It is shown in \cite{Riley} that
	\begin{equation}W_{11}+(1-\mu)W_{12}=0
	\label{nabPoly}
	\end{equation}
if and only if $m\mapsto N$, $b\mapsto C$ determines a well-defined representation of the knot group $\pi_K$. This is equivalent to $\rho_\epsilon$ being well-defined. 

To demonstrate (\ref{nabPoly}) we use the following. 

\begin{lem} Suppose $W=N^{\ve_1}C^{\ve_2}\ldots N^{\ve_{r-1}}C^{\ve_r}$ corresponds to a word $w$ in $\{m^{\pm1},b^{\pm1}\}$ under the assignment $m\mapsto N, b\mapsto C$, where $\ve_i=\pm1$. Then for all $n\in\ints$
	\begin{equation}
	[b^nw]\ee[b^n]W_{11}-W_{21}.
	\label{epsAndMatrix}
	\end{equation}
\label{LemepsAndW}
\end{lem}

We prove Lemma \ref{LemepsAndW} after the proof of Theorem \ref{Thm2bridge}. For any set $\set{\ve_1,\ve_2,\ldots,\ve_r}\subset\set{\pm1}^r$, call $N^{\ve_1}C^{\ve_2}\ldots N^{\ve_{r-1}}C^{\ve_r}$ a \emph{palindrome} if $\ve_i=\ve_{r+1-i}$ for $1\le i\le r$. Since $q$ is odd the parity of $\lfloor i\frac qp\rfloor$ is the same as that of $\lfloor q-i\frac qp\rfloor$ for $1\le i<p$, so $W$ is a palindrome. It is shown in \cite{Riley} that $-\mu uW_{12}=W_{21}$ if $W$ is a palindrome, which in our notation implies $([m]-[b])W_{12}\ee W_{21}$.

By using $[wm]\ee\mu[w]$ and applying the result of Lemma \ref{LemepsAndW}, with $n=1$ and $n=0$, we have
	\al{
[bw]-[mw]	&\ee [b] W_{11}-W_{21}-\mu\left([e] W_{11}-W_{21}\right)\\		&=\left([b]-\mu[e]\right)W_{11}-(1-\mu)W_{21}\\
&\ee([b]-[m])\left(W_{11}+(1-\mu)W_{12}\right).
	}

Since $bw=wm$ in $\pi_K$ and $\epsilon$ is well-defined, $\epsilon([bw])-\epsilon([wm])=0$. As $\epsilon([b])\ne\epsilon([m])$, we conclude that (\ref{nabPoly}) holds, and so $\rho_\epsilon$ is well-defined.

{\bf $\epsilon$ is induced by $\rho_{\epsilon}$:} It suffices to show that for any word $x$ in $\{m^{\pm1},b^{\pm1}\}$, if $X$ is the corresponding product of matrices $M^{\pm1}$ and $B^{\pm1}$, then
\begin{equation}
X_{11}=\langle Xe_1,e_1\rangle \ee \frac{[x]}{1-\mu}.
\label{EqnXinduction}
\end{equation} 

That (\ref{EqnXinduction}) holds for $x=m^{\pm1}$ follows immediately from $[m]\ee\mu(1-\mu)$. We claim that for $n\in\bb Z$, the entries of $B^n$ satisfy
	\al{
		(1-\mu)(B^n)_{11} 	&\ee [b^n]; \\
		(1-\mu)^2(B^n)_{12} &\ee ([b]-[e])([m^n]-[b^n])
	}
and so (\ref{EqnXinduction}) holds for $x=b^n$ for any $n\in\bb Z$.

\begin{proof}[Proof of claim] The case $n=1$ is apparent by definition and $n=-1$ follows from Lemma \ref{Lem2bridge} (i) and the fact that $\det B = \mu\ $ (note that $-\mu^{-1}([m]-[b]) = [m^{-1}]-[b^{-1}]$). 

Suppose both equalities hold for $n=k>0$. Then considering the product $B^kB$,
		\[(B^{k+1})_{11} \ee \frac{[b^k][b]}{(1-\mu)^2} + \frac{([b]-[e])([m^k]-[b^k])}{(1-\mu)^2} \ee \frac{[b^{k+1}]}{1-\mu},\]
the last equality from Lemma \ref{Lem2bridge} (ii). A similar calculation and another application of Lemma \ref{Lem2bridge} (ii) shows that $(B^{k+1})_{12} \ee \frac{([b]-[e])([m^{k+1}]-[b^{k+1}])}{(1-\mu)^2}$.

In similar fashion, if the claim holds for $n=k<0$ then taking the product $B^kB^{-1}$ and using Lemma \ref{Lem2bridge} (ii) we see it must hold for $n=k-1$.
\end{proof}

We can proceed inductively. As (\ref{EqnXinduction}) holds for $x=b^n$ we may suppose $m^{\pm1}$ appears in $x$. Assume (\ref{EqnXinduction}) holds for words of shorter length than $x$. 

Consider $m^{\delta}$, $\delta=\pm1$, and $x',x''$ such that $x=x'm^{\delta}x''$ as words in $\{m^{\pm1},b^{\pm1}\}$. Letting $E_{11}$ be the 2$\times$2 matrix with 1 in the top-left entry and 0 elsewhere, $M^{\delta}=I+(\mu^{\delta}-1)E_{11}$. Hence
\[X=X'M^{\delta}X''=X'X''+(\mu^{\delta}-1)X'E_{11}X'',\]
the $(1,1)$-entry of which is $(X'X'')_{11}+(\mu^{\delta}-1)(X')_{11}(X'')_{11}\ee\frac{[x'x'']}{1-\mu}-\delta\mu^{\frac12(\delta-1)}\frac{[x'][x'']}{1-\mu}$, by induction. The value of this expression and of $\frac{[x]}{1-\mu}$ under $\epsilon$ are equal by (\ref{SignIndptReln}). Hence $\rho_\epsilon$ induces $\epsilon$, completing the proof of Theorem \ref{Thm2bridge}.
\end{proof}

\begin{proof}[Proof of Lemma \ref{LemepsAndW}]
	The statement is trivially true when $w=1$ and $W=I$. Now let $W=N^{\ve_1}C^{\ve_2}W'$ and suppose that (\ref{epsAndMatrix}) holds for $W'$. Define $\delta_i=\frac12(\ve_i-1)$, for $1\le i\le r$. Then
	\al{
	[b^nw]	&\ee[b^{n+\ve_2}w']-\ve_1\mu^{\delta_1}[b^n][b^{\ve_2}w']\\
&\ee[b^{n+\ve_2}]W'_{11}-W'_{21}-\ve_1\mu^{\delta_1}[b^n]([b^{\ve_2}]W'_{11}-W'_{21})\\
&=\left([b^{n+\ve_2}]-\ve_1\mu^{\delta_1}[b^n][b^{\ve_2}]\right)W'_{11}-\left(1-\ve_1\mu^{\delta_1}[b^n]\right)W'_{21}.
}
	Now, we have that $N^{\ve_1}C^{\ve_2}=\begin{pmatrix}\mu^{\ve_1+\ve_2}+\ve_1\ve_2\mu^{\delta_1+\delta_2}\epsilon([m]-[b])&\ve_1\mu^{\delta_1}\\ \ve_2\mu^{\delta_2}\epsilon([m]-[b]) & 1\end{pmatrix}$, and so
	\[W_{11}\ee\left(\mu^{\ve_1+\ve_2}+\ve_1\ve_2\mu^{\delta_1+\delta_2}([m]-[b])\right)W'_{11}+\ve_1\mu^{\delta_1}W'_{21}\]
	and
\[W_{21}\ee\left(\ve_2\mu^{\delta_2}([m]-[b])\right)W'_{11}+W'_{21}.\]
Thus, to conclude the proof of the lemma, we need to show that
\begin{equation}
[b^n]\left(\mu^{\ve_1+\ve_2}+\ve_1\ve_2\mu^{\delta_1+\delta_2}([m]-[b])\right)-\ve_2\mu^{\delta_2}([m]-[b]) \ee[b^{n+\ve_2}]-\ve_1\mu^{\delta_1}[b^n][b^{\ve_2}].
\label{2BrEqn}
\end{equation}

Rearranging the expression in Lemma \ref{Lem2bridge} (ii), we know $[b^{n+1}]\ee \mu[b^n] + (1-\mu)[b^n] + \mu^n[b]-\mu^n[e]$. We can then derive that \begin{equation}
[b^{n+1}]\ee\mu[b^n]-([m]-[b])
\label{Impteqn}
\end{equation}

for all $n\ge0$ (applying strong induction to express $(1-\mu)[b^n]$ in terms of $[b]$ and $\mu$ produces a telescoping sum). We show that (\ref{2BrEqn}) holds in four cases.

{\bf Case $\ve_1=1, \ve_2=1$:} In this case $\delta_1=\delta_2=0$, so the left side of (\ref{2BrEqn}) is $[b^n]\left(\mu^2+[m]-[b]\right)-([m]-[b])$ which equals $\mu[b^n]-[b^n][b]-([m]-[b])$. Using (\ref{Impteqn}), we see that this equals the right side of (\ref{2BrEqn}).

{\bf Case $\ve_1=-1, \ve_2=1$:} Here, $\delta_1=-1$ and so the left side is $[b^n]\left(\mu+\mu^{-1}[b]\right)-([m]-[b])$ which equals $[b^{n+1}]+\mu^{-1}[b^n][b]$ by another application of (\ref{Impteqn}).

{\bf Case $\ve_1=1, \ve_2=-1$:}\\
In this case the left side is $[b^n]\left(\mu+\mu^{-1}[b]\right)+\mu^{-1}([m]-[b])$. On the right side we have 
{\small
\[[b^{n-1}]-[b^n][b^{-1}]\ee[b^{n-1}]-\mu^{-1}[b^n](1-\mu^2-[b])=\mu[b^n]+\mu^{-1}[b^n][b]+[b^{n-1}]-\mu^{-1}[b^n].\]
}
This provides us the result since $[b^{n-1}]-\mu^{-1}[b^n]\ee\mu^{-1}([m]-[b])$ by (\ref{Impteqn}).

{\bf Case $\ve_1=-1, \ve_2=-1$:} Finally, the left side in this case is \[[b^n]\left(\mu^{-2}+\mu^{-2}([m]-[b])\right)+\mu^{-1}([m]-[b])\ee[b^n]\left(\mu^{-1}+\mu^{-1}[b^{-1}]\right)+\mu^{-1}([m]-[b]).\]
The right side is $[b^{n-1}]+\mu^{-1}[b^n][b^{-1}]$. But (\ref{Impteqn}) says $\mu^{-1}[b^n]+\mu^{-1}([m]-[b])\ee[b^{n-1}]$, showing that equality (\ref{2BrEqn}) holds. This proves Lemma \ref{LemepsAndW}.
\end{proof}

\section{Augmentations and \augreps in higher dimensions}
\label{SecTorusknots}

We advance to knots for which $\wt{A}_K(\lambda,\mu)\ne A_K^2(\lambda,\mu)$, in which case $\aug_K(\lambda,\mu)$ has more factors than those in the classical $A$-polynomial. We tackle $(p,q)$-torus knots first, showing their knot group admits KCH irreps up to degree $\min(p,q)$. We then consider a family of 3-bridge pretzel knots, and find KCH irreps with degree three.

\subsection{Torus knots}
\label{subsecTorusKnots}

Let $T(p,q)$ be the torus knot for a coprime pair $(p,q)$ of positive integers. Here we prove Theorem \ref{Thmtorusknot}. As mentioned in the introduction, from this result and Theorem \ref{ThmDimBound} we get a new proof that meridional rank equals bridge number for torus knots (cf. \cite{RZ}). 

Let $\pi$ be the knot group of $T(p,q)$. We will work with the familiar 2-generator presentation, $\pi\cong\langle x,y \mid x^p=y^q\rangle$. Here the peripheral elements $m$ and $\ell$, representing a meridian and 0-framed longitude of $T(p,q)$ respectively, satisfy $\ell=x^pm^{-pq}$ and $m=x^sy^r$ where $r,s$ are integers such that $rp+sq=1$. Without any loss of generality we assume $\min(p,q)=p\ge 1$.

We need the following lemmas to prove Theorem \ref{Thmtorusknot}

\begin{lem}
Let $\wt{Y}$ be an $n\times n$ complex matrix of the form
		\[\wt{Y}=
		\begin{pmatrix}x_1 &x_2 & \ldots & x_{n-1} & y_0\\
		  1 & 0 & \ldots & 0 & y_1\\
		  0	& 1 & 		 & 0 & y_2\\
		  \vdots&  &\ddots  &\vdots&\vdots\\
		  0	& 	&0	&1&y_{n-1}
		\end{pmatrix}.\]
 Denote by $a(t)=a_0+a_1t+\cdots+a_{n-1}t^{n-1}+t^n$ the characteristic polynomial $\det(tI-\wt{Y})$. Then for all $1\le i\le n$,
			\[a_{n-i}=-x_i-y_{n-i}+\sum_{j=1}^{i-1}x_jy_{n-i+j}\]
		where we let $x_n=0$.
\label{LemCharPoly}
\end{lem}

The proof of Lemma \ref{LemCharPoly} is a routine calculation that we leave to the reader.

\begin{lem} Given $\mu_0\in\bb C^*$ and $1\le n\le p$, let $z$ be any $n^{th}$ root of $\mu_0^{pq}$ if $n<p$ and set $z=(-1)^{n-1}\mu_0^q$ if $n=p$. Then there exist $\zeta_1,\ldots,\zeta_n$, which are pairwise distinct $q^{th}$ roots of $z$ with product $\zeta_1\cdots\zeta_n = \mu_0^p$. There also exist $\eta_1,\ldots,\eta_n$, which are pairwise distinct $p^{th}$ roots of $z$ with product $\eta_1\cdots\eta_n = \mu_0^q$.
\label{LemDistinctRoots}
\end{lem}
\begin{proof}
	Consider the case $n<p$. Letting $\abs x$ denote the magnitude of $x\in\bb C$ and $\arg(x)$ its argument (mod $2\pi$), choose $k$ with $0\le k<n$ so that $\arg(z)=\frac{pq}n\arg(\mu_0)+\frac{2\pi k}n$. Each $\zeta_i$ is defined by its argument, with the understanding that $\abs{\zeta_i}=\abs z^{1/q}$. Let $\arg(\zeta_1)=\frac{\arg(z)}q-\frac{2\pi k}q$. If $n$ is odd, define $\zeta_2,\ldots,\zeta_n$ to have arguments distinct from $\zeta_1$, and differing from $\arg(z)/q$ by a multiple of $2\pi/q$, such that
	\begin{align}
	\arg(\zeta_i\zeta_{i+1})=\frac {2\arg(z)}q,\text{ when $i\ge 2$ is even.}\label{ZetaArgs}
	\end{align}
	When $n$ is even define $\zeta_2,\ldots,\zeta_{n-1}$ as in the odd case, and take $\arg(\zeta_n)=\arg(z)/q$. 

	By (\ref{ZetaArgs}) we have $\arg(\zeta_1\ldots\zeta_n)=\frac{n\arg(z)}q-\frac{2\pi k}q=\arg(\mu_0^p)$, and so $\zeta_1\ldots\zeta_n=\mu_0^p$. As $n<p$ we have $n\le q-2$. This allows the arguments of $\zeta_1,\ldots,\zeta_n$ to be chosen pairwise distinct: we only need to avoid the arguments $\frac{\arg(z)}q+\frac{2\pi k}q$ and $\frac{\arg(z)}q+\pi(n-1)$. Thus we may choose $\zeta_1,\ldots,\zeta_n$ as distinct $q^{th}$ roots of $z$.

	To choose $\eta_1,\ldots,\eta_n$ as $p^{th}$ roots of $z$, follow a similar procedure, at least when $n\le p-2$ (here the $\eta_i$ have magnitude $\abs z^{1/p}$ and arguments differ from $\arg(z)/p$ by a multiple of $2\pi/p$). If $n=p-1$, and $p$ is even, define $\eta_i$ for each $1\le i\le n$ to be determined by the set of $p-1$ arguments $\setn{(\arg(z)+2\pi j)/p}{1\le j\le p, j\not\equiv k-p/2\mod p}$. With this choice we have that $(\eta_i)^p = z$ for each $i$ and, 
		\begin{equation*}
		\arg(\eta_1\ldots\eta_n) = \frac{n\arg(z)+(p-2k)\pi}p+\pi = \frac{n\arg(z)-2\pi k}p = \arg(\mu_0^q),
		\end{equation*} 

	which shows that $\eta_1\ldots\eta_n = \mu_0^q$.

	To check the claim when $n=p$, a similar argument can be carried out for $z=(-1)^{n-1}\mu_0^q$. However, we are not free to change the argument of $z$ by altering $k$. If $\eta_1,\ldots,\eta_p$ are $p$ distinct $p^{th}$ roots of $z$ and $\eta_1\ldots\eta_p=\mu_0^q$, then $\arg(z)+(p-1)\pi = \arg(\mu_0^q)$.
\end{proof}

\begin{lem}
Let $\mu_0\in\bb C^*$. If $1\le n\le p<q$, there exists $X, Y\in\text{GL}_n\bb C$ such that
		\en{
			\item[(a)] for $r,s$ with $rp+sq=1$, $X^sY^r=M$, where $M$ is the $n\times n$ diagonal matrix $\text{diag}[\mu_0,1,\ldots,1]$;
			\item[(b)]$X^p=Y^q=z I$, where $I$ is the identity matrix, and $z$ is any $n^{th}$ root of $\mu_0^{pq}$ if $n<p$ and $z=(-1)^{n-1}\mu_0^q$ when $n=p$.
		}
\label{LemTorusKnot}
\end{lem}

\begin{proof}
	Define $c^Y(t)=\prod_{j=1}^n(t-\zeta_j^{-r})=c_0^Y+\cdots+c_{n-1}^Yt^{n-1}+t^n$ with $\zeta_1,\ldots,\zeta_n$ given by Lemma \ref{LemDistinctRoots}. Let $\wt Y$ be a matrix of the form in Lemma \ref{LemCharPoly}. We would like to choose entries $y_i$ of $\wt{Y}$ so that the characteristic polynomial $a(t)$ agrees with $c^Y(t)$. But Lemma \ref{LemCharPoly} shows that $a_{n-i}$ is linear in $y_{n-i}$ and independent of $y_l$ if $l<n-i$. Thus we may set 
		\[y_{n-1}=-c_{n-1}^Y-x_1,\]
	then recursively equate $y_{n-i}=-c_{n-i}^Y-x_i+\sum_{j=1}^{i-1}x_jy_{n-i+j}$ for $1<i\le n$. The resulting matrix $\wt{Y}$ has $c^Y(t)$ as its characteristic polynomial.
	
	For $\eta_1,\eta_2,\ldots,\eta_n$ as in Lemma \ref{LemDistinctRoots}, define $c^X(t)=\prod_{j=1}^n(t-\eta_j^s)=c_0^X+\cdots+c_{n-1}^Xt^{n-1}+t^n$. Setting $\wt{X}=M\wt{Y}$ with the $y_{n-i}$ determined as above, consider the formula for $y_{n-i}$, $1\le i\le n$. Note that $y_{n-i}$ is linear in $x_i$ and is independent of $x_l$ for $l>i$. Thus, if we replace $\wt{X}$ for $\wt{Y}$ in Lemma \ref{LemCharPoly} we see that the coefficient of $t^{n-i}$ in the characteristic polynomial of $\wt{X}$ has coefficients
	\[b_{n-i}= -(\mu_0-1)x_i +b_i'\]
	where $b_i'$ is a polynomial in variables $x_l$ with $l<i$ (here we have $b_1'=c_{n-1}^Y$).
	Since $\mu_0\ne1$, we may set $x_1=(-c_{n-1}^X+b_1')/(\mu_0-1)$, and recursively solve for each $x_i,$ for $1<i\le n-1$ so that $b_{n-i}=c_{n-i}^X$. Using the equations $\wt{X}=M\wt{Y}$, $\zeta_1\ldots\zeta_n = \mu_0^p$, and $\eta_1\ldots\eta_n=\mu_0^q$ we know that $b_0=c_0^X$ since 
	\[\det\wt{X}=\mu_0\det\wt{Y}=\mu_0(-1)^n(\prod_{j=1}^n\zeta_j^{-r})=(-1)^n\mu_0^{1-rp}=(-1)^n\mu_0^{sq}=(-1)^n\prod_{j=1}^n\eta_j^s.\]
	As a result, upon setting the $x_i$ equal to their solution, the characteristic polynomial of $\wt{X}$ is $c^X(t)$.
	
	The eigenvalues of $\wt{X}$ are pairwise distinct, so $\wt{X}$ is diagonalizable. Find $P$ such that $\wt{X}=P^{-1}\wt{D}P$, where $\wt{D}=\text{diag}[\eta_1^s,\eta_2^s,\ldots,\eta_n^s]$, then define $X=P^{-1}DP$ where $D=\text{diag}[\eta_1,\eta_2,\ldots,\eta_n]$. Define $Y$ in a similar manner, so that it is conjugate to $\text{diag}[\zeta_1,\zeta_2,\ldots,\zeta_n]$ by some $Q$ with $\wt{Y}=Q^{-1}\text{diag}[\zeta_1^{-r},\ldots,\zeta_n^{-r}]Q$.
	
	Now we have $X^s=P^{-1}D^sP=\wt{X}=M\wt{Y}=MY^{-r}$. Moreover, $X^p-z I=0$ and $Y^q-z I=0$ by the Cayley-Hamilton theorem.
\end{proof}

\vspace*{6pt}
\noindent{\bf Theorem \ref{Thmtorusknot}.}\ \ {\it Given $1\le p<q$, with $p,q$ relatively prime, let $T(p,q)$ denote the $(p,q)${--}torus knot and $\pi$ its knot group. For every $1\le n\le p$ and each $\mu_0\in\bb C^*$ there is a degree $n$ irreducible \augrep of $\pi$ with $\mu_0$ as an eigenvalue of the meridian $m$. In fact, 
	\[\wt{A}_{T(p,q)}=(\lambda\mu^{pq-q}+(-1)^p)\prod_{n=1}^{p-1}(\lambda^n\mu^{(n-1)pq}-1).\]}
\vspace*{6pt}

\begin{proof}
Considering the knot group for $T(p,q)$, the existence of $X,Y$ guaranteed by Lemma \ref{LemTorusKnot} provides a well-defined \augrep via $\rho(m)=M$, $\rho(x)=X$ and $\rho(y)=Y$. It follows $\rho(\ell)$ is the diagonal matrix $\text{diag}[\lambda_0,\ast,\ldots,\ast]$ with $z\mu_0^{-pq}=\lambda_0$, proving that $\wt{A}_{T(p,q)}$ has the desired factors.
	
That this accounts for all factors of $\wt{A}_{T(p,q)}$ is seen by showing that every KCH irrep of $\pi$ has this form. This is an application of Schur's Lemma. Note $x^p$ commutes with every element of $\pi$, since $x^py=y^{q+1}=yx^p$. Now an eigenspace $E_z$ of $\rho(x^p)$ must be an invariant subspace, so $E_z=\bb C^n$ and $\rho(x^p)$ is $z I$. Since $\ell$ is a commutator, $\det(\rho(\ell))=1$, so $\det(\rho(x^p))=\det(\rho(\ell m^{pq}))=\mu_0^{pq}$. Thus $z$ is an $n^{th}$ root of $\mu_0^{pq}$. 

That these representations are irreducible is deduced from Lemma \ref{KCHIrreps}. The bound on degree comes from Theorem \ref{ThmDimBound}.
\end{proof}

\subsection{\augreps for $(-2,3,2k+1)$ pretzel knots}
\label{subsecPretzelKnots}

In this section we find irreducible, degree three \augreps for the $(-2,3,2k+1)$ pretzel knots, which have a projection as shown in Figure \ref{FigPretzKnot}. Similar to the case of torus knots, the proof of irreducibility relies on Lemma \ref{KCHIrreps}. 

\begin{figure}[ht]
\begin{tikzpicture}[scale=.45,>=stealth]
	\draw[thick,->]
	(-1.5,5) -- node[at start,right]{{\footnotesize $m$}} (-1.75,3.75);	
	\draw[thick,->]
	(3.5,3.75) -- node[at end,right]{{\footnotesize $w$}} (4,6.25);
	\draw[blue]
	(-5.65,4) ..controls (-5.65,3.75) and (-6.65,3.5) .. (-6.65,3)
			  ..controls (-6.65,2.5) and (-5.65,2.25) .. (-5.65,2);
	\draw[blue,xscale=-1]
	(6-5.65,4-1) ..controls (6-5.65,3.75-1) and (6-6.65,3.5-1) .. (6-6.65,3-1)
			  ..controls (6-6.65,2.5-1) and (6-5.65,2.25-1) .. (6-5.65,2-1);
	\draw[blue]
	(12-5.65,4-3) ..controls (12-5.65,3.75-3) and (12-6.65,3.5-3) .. (12-6.65,3-3)
			  ..controls (12-6.65,2.5-3) and (12-5.65,2.25-3) .. (12-5.65,2-3);
	\draw[blue]
	(-5.65,2) ..controls (-5.65,-2.25) and (-.35,-2) .. (-.35,1);
	\draw[blue]
	(-0.35,3) ..controls (-0.35,3.25) and (0.65,3.5) .. (0.65,3.8);
	\draw[draw=white,very thick,double=blue]
	(.65,3.8) ..controls (0.65,4.3) and (5.35,4.3) .. (5.35,3.8);
	\draw[blue]
	(5.35,3.8)..controls (5.35,3.5) and (6.35,3.25) .. (6.35,3);
	\draw[blue]
	(6.35,-1) ..controls (6.35,-4.25) and (-6.35,-5.75) .. (-6.35,2);
	\draw[blue]
	(-6.35,2) ..controls (-6.35,2.15) .. (-6.1,2.3);
	\draw[blue]
	(-5.85,2.4)   ..controls (-5.35,2.75) .. (-5.35,3)
				   ..controls (-5.35,3.5) and (-6.35,3.75) ..(-6.35,4);
	\draw[draw=white,very thick,double=blue]
	(-6.35,4) ..controls (-6.35,6.25) and (6.35,6.25) .. (6.35,4);
	\draw[blue]
	(6.35,4) ..controls (6.35,3.75) and (5.35,3.5) ..(5.35,3);
	\draw[blue]
	(5.35,1)  ..controls (5.35,0.5) and (6.35,0.25) .. (6.35,0)
			  ..controls (6.35,-0.35) and (5.35,-.4)..(5.35,-1);
	\draw[blue]
	(5.35,-1) ..controls (5.35,-1.85) and (.65,-2.6) ..(.65,1);
	\draw[blue,xscale=-1]
	(5.35-6,1+2)  ..controls (5.35-6,0.5+2) and (6.35-6,0.25+2) .. (6.35-6,0+2)
			  ..controls (6.35-6,-0.35+2) and (5.35-6,-.4+2)..(5.35-6,-1+2);
	\draw[blue]
	(.65,3) ..controls (.65,3.5) and (-.35,3.75) ..(-.35,4);
	\draw[draw=white,very thick,double=blue]
	(-.35,4)..controls (-.35,4.6) and (-2,4.6) ..(-3,4.6);
	\draw[blue,->,thin]
	(-.35,4)..controls (-.35,4.6) and (-2,4.6) ..node[at end,above]{{\tiny $L$}} (-3,4.6);
	\draw[blue]
	(-3,4.6)..controls (-4,4.6) and (-5.65,4.6) ..(-5.65,4);
	\draw
	(-6.5,2) ..controls (-6.5,2.5) and (-5.5,2.5) ..(-5.5,3)
	(-6.5,3) ..controls (-6.5,3.5) and (-5.5,3.5) ..(-5.5,4);
	\draw[draw=white,double=black,very thick]
	(-6.5,3) ..controls (-6.5,2.5) and (-5.5,2.5)..(-5.5,2)
	(-6.5,4) ..controls (-6.5,3.5) and (-5.5,3.5)..(-5.5,3);
	\draw
	(-.5,2) ..controls (-.5,1.5) and (.5,1.5)..(.5,1)
	(-.5,3) ..controls (-.5,2.5) and (.5,2.5)..(.5,2)
	(-.5,4) ..controls (-.5,3.5) and (.5,3.5)..(.5,3);
	\draw[draw=white,double=black,very thick]
	(-.5,1) ..controls (-.5,1.5) and (.5,1.5) ..(.5,2)
	(-.5,2) ..controls (-.5,2.5) and (.5,2.5) ..(.5,3)
	(-.5,3) ..controls (-.5,3.5) and (.5,3.5) ..(.5,4);
	\draw
	(6-.5,0) ..controls (6-.5,-.5) and (6.5,-.5)..(6.5,-1)
	(6-.5,1) ..controls (6-.5,.5) and (6.5,.5)..(6.5,0)
	(6-.5,4) ..controls (6-.5,3.5) and (6.5,3.5)..(6.5,3);
	\draw[draw=white,double=black,very thick]
	(6-.5,-1) ..controls (6-.5,-.5) and (6.5,-.5) ..(6.5,0)
	(6-.5,0) ..controls (6-.5,.5) and (6.5,.5) ..(6.5,1)
	(6-.5,3) ..controls (6-.5,3.5) and (6.5,3.5) ..(6.5,4);
	\draw (6,2.25) node {{\footnotesize $\vdots$}};
	
	\draw[draw=white,very thick,double=black]
	(-5.5,4) ..controls (-5.5,4.5) and (-.5,4.5).. (-.5,4);
	\draw[draw=white,very thick,double=black]
	(.5,4) ..controls (.5,4.5) and (5.5,4.5).. (5.5,4);
	\draw[draw=white,very thick,double=black]
	(-6.5,4) ..controls (-6.5,6.5) and (6.5,6.5).. (6.5,4);
	\draw
	(-6.5,2) ..controls (-6.5,-6) and (6.5,-4.5).. (6.5,-1)
	(-5.5,2) ..controls (-5.5,-2) and (-.5,-1.75).. (-.5,1)
	(.5,1) ..controls (.5,-2.75) and (5.5,-2.1).. (5.5,-1);
	\draw[thick] 
	(7,4.5) ..controls (8,4.5) and (7,2).. (8,1.75)
	(7,-1) ..controls (8,-1) and (7,1.5).. (8,1.75);
	\draw (9.5,1.75) node {{\footnotesize $2k+1$}};
	
\end{tikzpicture}
\caption{The $(-2,3,2k+1)$ pretzel knot with a blackboard framed longitude $L$.}
\label{FigPretzKnot}
\end{figure}
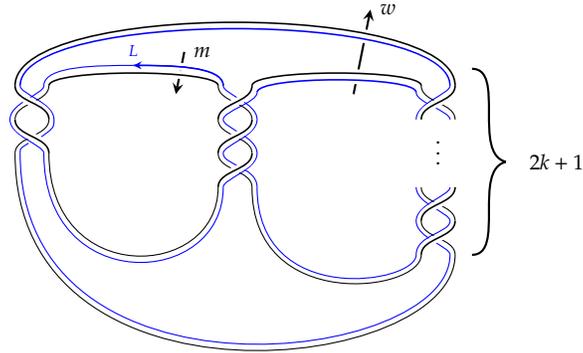

\vspace*{6pt}
\noindent{\bf Theorem \ref{ThmPretzels}.}\ \ {\it Let $K$ be the $(-2,3,2k+1)$ pretzel knot, where $k\ne -1,0$. Then there is a 3-dimensional, irreducible \augrep of $\pi_K$. Furthermore, $(1-\lambda\mu^{2k+6})$ divides $A_K^3(\lambda,\mu)$.}
\vspace*{6pt}

\begin{proof}
	The group $\pi_K$ admits a two-generator presentation $\pi_K=\langle{m,w \mid w^kE=Fw^k}\rangle$, where $m$ is a meridian, $E=mwm^{-1}w^{-1}m^{-1}$, $F=m^{-1}w^{-1}mwmw^{-1}$, and $w$ is a product of two meridians \cite[\S 4]{LT12}. As a consequence of Theorem \ref{ThmDimBound}, any irreducible \augrep that sends $m$ to $M=\text{diag}[\mu_0,1,\ldots,1]$ has degree 3 or less.

	Let us first observe what would be required of an augmentation $\epsilon:HC_0(K)\to\bb C$ that is induced from $\rho:\pi_K\to\text{GL}_3\bb C$, a KCH irrep. We will abuse notation and write $\mu$ for $\epsilon(\mu)=\mu_0$. By Remark \ref{Rem2genForm} in some basis we have $\rho(m)=\text{diag}[\mu,1,1]$ and
		\[\rho(w)=\begin{pmatrix}x_1&x&y_0\\ 1&0&y_1\\ 0&1&y\end{pmatrix}.\]
	As $\rho$ induces the augmentation we must have $x_1\ee\frac{[w]}{1-\mu}$. Note that $y_0-x_1y_1-xy=\det(\rho(w))=\mu^2$ since $w$ is the product of two conjugates of $m$. Furthermore, an examination of the upper left entry of $\rho(w)^{-1}$ requires that $\frac{[w^{-1}]}{1-\mu}\ee\rho(w^{-1})_{11}=-\frac{y_1}{\mu^2}$ in this basis. Thus, for our supposed representation, we have that in some basis and for some pair $x,y\in\bb C$,
		\begin{equation}
			\rho(w)=\begin{pmatrix}\frac{\epsilon([w])}{1-\mu}&x&\mu^2-\mu^2\frac{\epsilon([w])\epsilon([w^{-1}])}{(1-\mu)^2}+xy\\ 1 & 0 & -\mu^2\frac{\epsilon([w^{-1}])}{1-\mu}\\ 0&1&y\end{pmatrix}.
		\label{EqnWmatrix}
		\end{equation}

	By considering the relations (\ref{epsRel}) on the values of $\epsilon$ and that $w^k=Fw^kE^{-1}$ in $\pi_K$, we obtain
		\[[w^k]\ee\mu^{-2}([w^{-1}])^2[w^{k+1}]-2\mu^{-1}[w^k][w^{-1}]\left(1+\mu^{-1}[w][w^{-1}]\right)+[w^{k-1}]\left(1+\mu^{-1}[w][w^{-1}]\right)^2.\]
	Were $\epsilon$ to satisfy $\epsilon([w])=1$ and $\epsilon([w^{-1}])=-\mu$, then this relation would become $[w^k]\ee[w^{k+1}]$, implying that the (1,1)-entry of $\rho(w)^{k+1}$ would agree with that of $\rho(w)^k$.

	Motivated by these observations, let us define $W$ to be the matrix in (\ref{EqnWmatrix}) with $\epsilon([w])$ replaced by $1$ and $\epsilon([w^{-1}])$ replaced by $-\mu$. Denote by $E_0$ (and $F_0$) the matrix corresponding to $\rho(E)$ (and $\rho(F)$ respectively) in our basis.

	\begin{lem}Let $W$ be defined as above, but with $y$ set equal to $1+\mu+(\mu^{-1}-1)x$, and let $M=\text{diag}[\mu,1,1]$. Suppose that $(W^k)_{11}=(W^{k+1})_{11}$ and define $R=W^kE_0-F_0W^k$. Then $R=0$.
	\label{LemPretzelsRis0}
	\end{lem}

	Lemma \ref{LemPretzelsRis0} is proved below; let us first use it to complete the proof of Theorem \ref{ThmPretzels}. For a given $\mu\in\bb C^*$, $\mu\ne1$, note that if $y=1+\mu+(\mu^{-1}-1)x$ and $k\ne0,-1$ then $\Phi_k=(W^{k+1})_{11}-(W^k)_{11}$ is a polynomial in $x$ with positive degree. Thus it has a root in $\bb C$; set $x$ as a root. Lemma \ref{LemPretzelsRis0} implies that from $m\mapsto M$ and $w\mapsto W$ we obtain a well-defined homomorphism $\rho:\pi_K\to \text{GL}_3\bb C$.

	To see that $\rho$ is irreducible, we use the following lemma. 

	\begin{lem}
	If $\epsilon$ is the augmentation induced from the \augrep $\rho:\pi_K\to\text{GL}_3\bb C$ defined above and $\ell\in\pi_K$ the longitude of $K$, then $\epsilon(\lambda)=\mu^{-(2k+6)}$.
	\label{LemLongitude}
	\end{lem}

	From Lemma \ref{LemLongitude} we conclude that $(1-\lambda\mu^{2k+6})$ divides $A_K^3(\lambda,\mu)$. Were this factor to exist in $A^2_K(\lambda,\mu)$ there would be a factor of $(1-\lambda\mu^{4k+12})$ in the $A$-polynomial of the $(-2,3,2k+1)$-pretzel knot. Results in \cite[Theorem 1.6]{Mat02} (upon translating to when $n=2k+1$) show that the variety has one, or two, components, depending on whether 3 does not, or does, divide $k-1$. The geometric component appearing in both cases does not agree with $(1-\lambda\mu^{4k+12})$ and when $3\vert(k-1)$ the non-geometric factor in the $A$-polynomial is $(1-\lambda\mu^{4k+8})$. 

	By Lemma \ref{KCHIrreps}, $\rho$ cannot be reducible. The proof of Lemma \ref{LemLongitude} is presented after the proof of Lemma \ref{LemPretzelsRis0}.

\end{proof}

\begin{proof}[Proof of Lemma \ref{LemPretzelsRis0}]
	We begin by noting some identities between the entries of $W^k$. We sometimes use $y$ to denote $1+\mu+(\mu^{-1}-1)x$.
	
	Since $W^{k+1}=W^kW$ we see that $(W^{k+1})_{11} = \frac1{1-\mu}(W^k)_{11}+(W^k)_{12}$. Thus by our assumption, 
		\begin{equation}
		(W^k)_{12}=\frac{-\mu}{1-\mu}(W^k)_{11}.
		\label{eqn1}
		\end{equation}
		
	Now, upon noting that $E_0=MWM^{-1}W^{-1}M^{-1}$ and $F_0=M^{-1}W^{-1}MWMW^{-1}$, one may compute that
	\begin{equation}
		E_0 = \begin{pmatrix}\frac{-\mu}{1-\mu}&1+\frac{\mu}{(1-\mu)^2}&\frac{\mu}{1-\mu}x\\ \frac{-1}{\mu}&\frac 1{\mu(1-\mu)}&x\\ 0&0&1\end{pmatrix} \qquad F_0= \begin{pmatrix}1&0&0\\ 0&\frac{y}{\mu}&xy\frac{1-\mu}{\mu^2}+1\\ 0&\frac{-1}{\mu}&-x\frac{1-\mu}{\mu^2}\end{pmatrix}
	\label{EandF}
	\end{equation}

	Now write out the first row of $R$ in terms of the entries of $W^k$. In each case, an application of (\ref{eqn1}) shows that $R_{1j}=0$ for $j=1,2,3$.

	Using $W^{k+1}=WW^k$, and the supposition that $(W^{k+1})_{11}=(W^k)_{11}$, we also see that
		\begin{equation}
		\frac{\mu}{1-\mu}(W^k)_{11}+x(W^k)_{21}+\left(\mu^2+\frac{\mu^3}{(1-\mu)^2}+xy\right)(W^k)_{31}=0
		\label{eqn2}
		\end{equation}

	We also have the following two identities:
		\begin{equation}
		\frac{1}{1-\mu}(W^k)_{21}+(W^k)_{22}=(W^kW)_{21}=(WW^k)_{21}=(W^k)_{11}+\frac{\mu^3}{1-\mu}(W^k)_{31};
		\label{21}
		\end{equation}
		\begin{equation}
		\frac1{1-\mu}(W^k)_{31}+(W^k)_{32} = (W^kW)_{31}=(WW^k)_{31}=(W^k)_{21}+y(W^k)_{31}.
		\label{31}
		\end{equation}

	Additionally, note that $y=1+\mu+(\mu^{-1}-1)x$ implies that $\frac{1-y(1-\mu)-\mu^2}{1-\mu}+\frac{1-\mu}{\mu}x=0$. Now we may use (\ref{EandF}), (\ref{eqn2}), and (\ref{21}) to calculate that 
	\al{
		R_{21} 
		&=\frac{-\mu}{1-\mu}(W^k)_{21}+\frac{-1}{\mu}(W^k)_{22}-y\left(\frac 1{\mu}(W^k)_{21}+x\frac{1-\mu}{\mu^2}(W^k)_{31}\right)-(W^k)_{31}\\
		&= \frac{-\mu}{1-\mu}(W^k)_{21}+\frac{1}{\mu}\left((W^k)_{11}-(W^k)_{22}\right)-y\left(\frac 1{\mu}(W^k)_{21}+x\frac{1-\mu}{\mu^2}(W^k)_{31}\right)-(W^k)_{31}-\frac{1}{\mu}(W^k)_{11}\\
		&= \frac{1}{\mu}\left(\frac{1-y(1-\mu)-\mu^2}{1-\mu}\right)(W^k)_{21}-\frac{\mu^2}{1-\mu}(W^k)_{31}-(W^k)_{31}-\frac{1-\mu}{\mu^2}\left(\frac{\mu}{1-\mu}(W^k)_{11}+yx(W^k)_{31}\right)\\
		&= \frac{1}{\mu}\left(\frac{1-y(1-\mu)-\mu^2}{1-\mu}+\frac{1-\mu}{\mu}x\right)(W^k)_{21}+\left(1-\mu+\frac{\mu}{1-\mu}-\frac{\mu^2}{1-\mu}-1\right)(W^k)_{31}=0.
	}

	By a similar calculation, that uses (\ref{31}), we also see that $R_{31}=0$: 
	\al{
		R_{31} 
		&= \frac{-\mu}{1-\mu}(W^k)_{31}+\frac{-1}{\mu}(W^k)_{32}+\left(\frac{(W^k)_{21}}{\mu}+x\frac{1-\mu}{\mu^2}(W^k)_{31}\right)\\
		&= \mu^{-1}\left((W^k)_{21}-(W^k)_{32}\right)+\left(x\frac{1-\mu}{\mu^2}-\frac{\mu}{1-\mu}\right)(W^k)_{31}\\
		&= \mu^{-1}\left(\frac{1-y(1-\mu)}{1-\mu}(W^k)_{31}\right)+\left(x\frac{1-\mu}{\mu^2}-\frac{\mu}{1-\mu}\right)(W^k)_{31}\\
		&= \mu^{-1}\left(\frac{1-y(1-\mu)-\mu^2}{1-\mu}+\frac{1-\mu}{\mu}x\right)(W^k)_{31}=0.
	}

	Similar, though perhaps more extensive, calculations may be carried out for the remaining four entries of $R$, showing that each one vanishes. For these calculations, the assiduous reader will want to use identities
	\begin{equation}
	x(W^k)_{31}+(W^k)_{33}=(W^{k+1})_{32}=(W^k)_{22}+y(W^k)_{32}\qquad\text{and}
	\label{32}
	\end{equation}
	{\small
	\begin{equation}
	\left(\mu^2+\frac{\mu^3}{(1-\mu)^2}+xy\right)(W^k)_{31}+\frac{\mu^3}{1-\mu}(W^k)_{32}+y(W^k)_{33}=(W^{k+1})_{33}=(W^k)_{23}+y(W^k)_{33}
	\label{33}
	\end{equation}
	}

	in conjunction with (\ref{21}) and (\ref{31}) to reduce $(W^k)_{23}$, $(W^k)_{33}$, $(W^k)_{22}$, and $(W^k)_{32}$ to expressions in $(W^k)_{11}, (W^k)_{21}$, and $(W^k)_{31}$. This, combined with (\ref{eqn2}) yields the result.
\end{proof}

Having proved Lemma \ref{LemPretzelsRis0} we obtain a well-defined augmentation $\epsilon:HC_0(K)\to\bb C$ satisfying $\epsilon([w])=1$, $\epsilon([w^{-1}])=-\mu$, and $\epsilon([w^k])=\epsilon([w^{k+1}])$. We will use these properties to prove Lemma \ref{LemLongitude}.

We will need to use an auxiliary result near the end of the proof of Lemma \ref{LemLongitude}.

\begin{lem}
Let $j\ge0$. Then $[w^jmw]\ee\mu^{2j+1}[w^{-j-1}mw]$.
\label{AuxLemWpower}
\end{lem}
\begin{proof}
	One can readily use the relations on the cord algebra and the facts $[w]\ee1$ and $[w^{-1}]\ee -\mu$ to verify the lemma when $j=0$. The claim of the lemma concerns entries of a linear combination of powers of $W$. Indeed, $[w^jmw]=[w^{j+1}]-[w^j]$ and $[w^{-j-1}mw]=[w^{-j}]-[w^{-j-1}]$. Moreover, by the definition of $\epsilon$ we know \begin{equation}[w^{j+1}]-[w^j]\ee(1-\mu)(W^{j+1}_{11}-W^j_{11})\quad\text{and}\quad [w^{-j}]-[w^{-j-1}]\ee(1-\mu)(W^{-j}_{11}-W^{-j-1}_{11}).\nonumber\end{equation}
	As $\mu\ne1$, the statement to be proved is
		\begin{equation}
		A_j:=W_{11}^{j+1}-W_{11}^{j}-\mu^{2j+1}(W_{11}^{-j}-W_{11}^{-j-1})=0
		\label{LemWpowRelation}
		\end{equation}

	We know that $A_0=0$. Except for finitely many values of $\mu$, one checks that $W$ is diagonalizable with eigenvalues $1, a-\sqrt b, a+\sqrt b$ (for some square root of $b$). Here $a,b$ are elements\footnote{$a=\frac{x(1-\mu)^2+\mu(2-\mu^2)}{2\mu(1-\mu)}-\frac12$} of $\rat(\mu,x)$ such that $a^2-b=\mu^2$. Let $S$ be the matrix of eigenvectors so that $S^{-1}WS=\text{diag}[1,a-\sqrt b,a+\sqrt b]$.

	Using the diagonalization, we compute that there are numbers\footnote{$A=\frac{(a-1-\sqrt{b})(a-\mu-\sqrt{b})(x(1-\mu)^2-\mu(a-a\mu+\mu^2-\sqrt{b}-\mu(2-\sqrt{b})))}{2\sqrt{b}(x(1-\mu)^2(2\mu-1)+\mu^2(\mu^3-4\mu^2+(2a+5)\mu-(2a+1))},\quad B=\frac{(a-1+\sqrt{b})(a-\mu+\sqrt{b})(x(1-\mu)^2-\mu(a-a\mu+\mu^2+\sqrt{b}-\mu(2+\sqrt{b})))}{2\sqrt{b}(x(1-\mu)^2(2\mu-1)+\mu^2(\mu^3-4\mu^2+(2a+5)\mu-(2a+1))}$} $A,B\in\rat(\mu,x,\sqrt b)$, that are independent of $j$, such that
		\[A_j=\left(a-\sqrt b\right)^{-j-1}\left(\mu^{2j+1}-(a-\sqrt b)^{2j+1}\right)A+\left(a+\sqrt b\right)^{-j-1}\left(\mu^{2j+1}-(a+\sqrt b)^{2j+1}\right)B.\]
	As $\mu^2=(a+\sqrt b)(a-\sqrt b)$ we have that
		{\small
		\al{
			A_j	&=\left(a-\sqrt b\right)^{-j}\left((a+\sqrt b)\mu^{2j-1}-(a-\sqrt b)^{2j}\right)A+\left(a+\sqrt b\right)^{-j}\left((a-\sqrt b)\mu^{2j-1}-(a+\sqrt b)^{2j}\right)B\\
			&=\left(a-\sqrt b\right)\left(\left(a-\sqrt b\right)^{-j}\left(\mu^{2j-1}-(a-\sqrt b)^{2j-1}\right)A+\left(a+\sqrt b\right)^{-j}\left(\mu^{2j-1}-(a+\sqrt b)^{2j-1}\right)B\right)\\
			&\qquad+\left(a-\sqrt b\right)^{-j}\mu^{2j-1}(2\sqrt b)A+\left(a+\sqrt b\right)^{-j}(2\sqrt b)(a+\sqrt b)^{2j-1}B\\
			&=(a-\sqrt{b})A_{j-1}+\left(a+\sqrt b\right)^{j}\left(\mu^{-1}(2\sqrt b)A+(2\sqrt b)(a+\sqrt b)^{-1}B\right).
		}
		}
	The numbers $A,B,a,b$ are such that 
	$\mu A+(a-\sqrt b)B=0,$ and this implies that $\mu^{-1}(2\sqrt b)A+(2\sqrt b)(a+\sqrt b)^{-1}B=0$. By induction $A_{j-1}=0$ and so $A_j=0$, concluding the proof.
\end{proof}

We collect some additional observations before proving Lemma \ref{LemLongitude}.

\begin{lem} The matrix $F_0$ commutes with $M$ and $E_0W=F_0^{-1}$. We also have, for any $g,h\in\pi_K$, the following identities on values of $\epsilon$:
	\begin{align}
	[gEwh]	&\ee[gF^{-1}h];\label{Id1}\\
	[F^{\pm1}g]&\ee[g]\ee[gF^{\pm1}];\label{Id2}\\
	[gE^{-1}]&\ee[gw].\label{Id3}
	\end{align}
Moreover, $[gE^imw]\ee\mu^{-i}[gmw]$ and $[E^{-1}mE^ig]\ee\mu^{-i+1}([g]-[w^{-1}g])$ for any $g\in\pi_K$ and any integer $i$.
\label{LemComphelp}
\end{lem}

\begin{proof}
	We readily check $F_0$ commutes with $M$ and $E_0W=F_0^{-1}$ from (\ref{EandF}). Identity (\ref{Id1}) is a consequence of $E_0W=F_0^{-1}$. That (\ref{Id2}) holds is inferred from the block form of $F_0$, and (\ref{Id3}) from the observation that $W$ and $E_0^{-1}$ have the same first column.
	
	To see that $[gE^imw]\ee\mu^{-i}[gmw]$ for any $g\in\pi_K$, note that $E^i=mwm^{-i}w^{-1}m^{-1}$ and use that $[gmwm^{-i}]\ee\mu^{-i}[gmw]$. Finally,  since $[E^{-1}]\ee[w]\ee1$, we know that for $g\in\pi_K$
	\al{
		[E^{-1}mE^ig] &\ee[E^{i-1}g]-[E^ig]\\
			  &\ee-[mw][m^{-i}w^{-1}m^{-1}g]\\
			  &\ee-\mu[w](\mu^{-i}[w^{-1}m^{-1}g])\\
			  &\ee-\mu^{-i+1}([w^{-1}m^{-1}][g]+[w^{-1}g])\\
			  &\ee\mu^{-i+1}([g]-[w^{-1}g]).
	}
\end{proof}

\begin{proof}[Proof of Lemma \ref{LemLongitude}]

	In Figure \ref{FigPretzKnot} we depict a projection of $K$ and a longitude of $K$ with ``blackboard framing.'' Denote the longitude there depicted by $L$; we have $\ell=m^{-(2k+6)}L$ in $\pi_K$. We will show that $[L]\ee 1-\mu$. As a result $\lambda(1-\mu)\ee [\ell]=[m^{-(2k+6)}L]\ee \mu^{-(2k+6)}(1-\mu)$, and so $\lambda = \mu^{-(2k+6)}$.

	A computation from the projection in Figure \ref{FigPretzKnot} allows us to write $L$ in terms of $m,w,$ and $E$:
	{\small	
		\[L = (Ew)(E^{-1}mE)E^{k}w^{k}((Ew)^{-1}mEw)(E^{-1}mE)(Ew)^{-k}w^{k}((Ew)^{-1}mEw)E^{-1}.\]
	}

	Using (\ref{Id1}) and that $F_0$ and $M$ commute from Lemma \ref{LemComphelp}, we may replace $(Ew)^{-1}m(Ew)$ with $m$ in the preceding expression without changing the value under the map $\epsilon$. From this observation and (\ref{Id2}) we obtain
		\[[L]\ee [E^{-1}mE^{k+1}w^{k}(mE^{-1}mE)F^kw^{k}mE^{-1}].\]

	Now use the relation $w^kE=Fw^k$ in $\pi_K$ and (\ref{Id3}) to obtain
		\[[L]\ee [E^{-1}mE^{k+1}w^{k}(mE^{-1}mE)w^{k}E^kmw].\]

	The facts $[gE^imw]\ee\mu^{-i}[gmw]$ and $[E^{-1}mE^ig]\ee\mu^{-i+1}([g]-[w^{-1}g])$ for any $g\in\pi_K$ inform us that
		\[[L]\ee\mu^{-2k}([w^{k}(mE^{-1}mE)w^{k}mw]-[w^{k-1}(mE^{-1}mE)w^{k}mw]).\]
	By applying relations in (\ref{epsRel}), Lemma \ref{LemComphelp} and using that $[w^kmw]=[w^{k+1}]-[w^k]\ee 0$,
	\al{
	[L]	&\ee \mu^{-2k}([w^{k}(mE^{-1}mE)w^{k}mw]-[w^{k-1}(mE^{-1}mE)w^{k}mw])\\
	 	&\ee \mu^{-2k}([w^{k}mw^{k}mw]-[w^kmE^{-1}][Ew^kmw]-[w^{k-1}mw^{k}mw]+[w^{k-1}mE^{-1}][Ew^kmw])\\
	 	&\ee \mu^{-2k}([w^{k}mw^{k}mw]-[w^{k-1}mw^{k}mw]+[w^{k-1}mw][w^{k-1}mw])\\
	 	&\ee \mu^{-2k-1}([w^{k}mw^{k}E^{-1}mw]-[w^{k-1}mw^{k}E^{-1}mw]+[w^{k-1}mw][w^{k-1}E^{-1}mw])\\
	 	&\ee \mu^{-2k-1}([w^{k}mF^{-1}w^{k}mw]-[w^{k-1}mF^{-1}w^{k}mw]+[w^{k-1}mw][w^{-1}F^{-1}w^{k}mw])\\
	 	&\ee \mu^{-2k-1}([w^{k+1}mw^{k}mw]-[w^{k}mw^kmw]).
	}

	We will calculate that \begin{equation}[w^{k+1}mw^kmw]\ee -\mu^{2k+2}\quad\text{and}\quad [w^kmw^kmw]\ee -\mu^{2k+1}.
	\label{EpsCalc4L}
	\end{equation} 
	Our calculation on $\epsilon([L])$ then implies that $[L]\ee1-\mu$. We compute (\ref{EpsCalc4L}) by showing that for any $i\ge0$, 
	\begin{equation}
	\mu^{-i}[w^{k+i}mw]+\mu^i[w^{k-i}mw]\ee 0,
	\label{LemWpower2}
	\end{equation}

	Taking $i=k$ we see that $0\ee\mu^{-k}[w^{2k}mw]+\mu^k[mw]$ which implies $[w^kmw^kmw]\ee[w^{2k}mw]\ee -\mu^{2k+1}$, the first equality here because $[w^kmw]\ee0$. By taking $i=k+1$ we get that $[w^{2k+1}mw]=-\mu^{2k+2}[w^{-1}mw]=-\mu^{2k+2}$.

	Turning to the proof of (\ref{LemWpower2}), note that when $i=0$, both sides are zero. For the case $i=1$ we have
	\al{
	\mu^{-1}[w^{k+1}mw]
	&+\mu[w^{k-1}mw]\\
	&\ee [w^{k-1}E^{-1}mw]+\mu^{-1}[wmw^{k+1}]\\
	&\ee [w^{k-1}E^{-1}w]-[w^k]-\mu^{-1}([m^{-1}w^{-1}][wmw^{k+1}]+[mw^{k+1}])+[w^{k+1}]\\
	&\ee [w^{k+1}]-[w^k]+[w^{k-1}E^{-1}w]-[w^{-1}F^{-1}w^{k+1}]\\
	&\ee[w^{k+1}]-[w^k]\ee 0,
	}
	the second to last equality from $w^kE^{-1}=F^{-1}w^k$. We see that (\ref{LemWpower2}) holds for $i=1$.

	Given some $i>1$, if (\ref{LemWpower2}) holds for $j<i$ then
	\al{
	\mu^{-i}[w^{k+i}mw]+\mu^i[w^{k-i}mw]
	&\ee \mu^{i-1}[w^{k-i}E^{-1}mw]+\mu^{-i}[wmw^{k+i}]\\
	&\ee\mu^{i-1}[w^{k-i}E^{-1}mw]-\mu^{-i}\left([w^{-1}m^{-1}][wmw^{k+i}]+[mw^{k+i}]\right)+\mu^{-i+1}[w^{k+i}]\\
	&\ee\mu^{i-1}[w^{k-i}E^{-1}w]-\mu^{i-1}[w^{k-i+1}]+\mu^{-i+1}[w^{k+i}]-\mu^{-i+1}[w^{-1}F^{-1}w^{k+i}]\\
	&=\mu^{i-1}[w^{-i}F^{-1}w^{k+1}]-\mu^{i-1}[w^{k-i+1}]+\mu^{-i+1}[w^{k+i}]-\mu^{-i+1}[w^{k-1}E^{-1}w^i]\\
	&\ee \mu^{i-1}[w^{k-i+1}mw]+\mu^{-i+1}[w^{k+i-1}mw]\\
	&\qquad+\mu^{i-1}[w^{-i+1}m^{-1}w^{-1}m^{-1}][wmw^{k+1}]+\mu^{-i+1}[w^{k-1}mw][w^{-1}m^{-1}w^i]\\
	&\ee\mu^{i-1}[w^{-i+1}m^{-1}w^{-1}m^{-1}][wmw^{k+1}]+\mu^{-i+1}[w^{k-1}mw][w^{-1}m^{-1}w^i],
	}
	the last equality from the induction hypothesis. Now $[w^{-i+1}m^{-1}w^{-1}m^{-1}]\ee-\mu^{-1}[w^{-i}mw]$ and also $[w^{-1}m^{-1}w^i]\ee-[w^{i-1}mw]$. But as was shown in Lemma \ref{AuxLemWpower}, $[w^{i-1}mw]\ee\mu^{2i-1}[w^{-i}mw]$ for all $i$. And so  
	\al{
	\mu^{-i}[w^{k+i}mw]+\mu^i[w^{k-i}mw] 
	&\ee-\mu^{i-1}[w^{-i}mw](\mu^{-1}[w^{k+1}mw])-\mu^{-i+1}[w^{k-1}mw](\mu^{2i-1}[w^{-i}mw])\\
	&\ee-\mu^{i-1}[w^{-i}mw]\left(\mu^{-1}[w^{k+1}mw]+\mu[w^{k-1}mw]\right)\ee0,
	}
	since (\ref{LemWpower2}) holds when $i=1$.
\end{proof}
 
\bibliography{KCHReps_refs}
\bibliographystyle{alpha}

\end{document}